\documentclass[12pt,english]{extarticle}

\usepackage[latin9]{inputenc}
\usepackage{color}
\usepackage{babel}
\usepackage{mathrsfs}
\usepackage{amsmath}
\usepackage{amsthm}
\usepackage{amssymb}
\usepackage{geometry}
\usepackage{setspace}
\usepackage[pdfusetitle,
 bookmarks=true,bookmarksnumbered=false,bookmarksopen=false,
 breaklinks=false,pdfborder={0 0 1},backref=false,colorlinks=true]
 {hyperref}
\hypersetup{
 pdfborderstyle=,linkcolor=blue,urlcolor=cyan,citecolor=red}

\makeatletter
\numberwithin{equation}{section}
\numberwithin{figure}{section}
\numberwithin{table}{section}
\usepackage{paralist}
\theoremstyle{plain}
\newtheorem{thm}{\protect\theoremname}[section]
\theoremstyle{definition}
\newtheorem{defn}[thm]{\protect\definitionname}
\theoremstyle{definition}
\newtheorem{rem}[thm]{\protect\remarkname}
\theoremstyle{plain}
\newtheorem{prop}[thm]{\protect\propositionname}
\theoremstyle{plain}
\newtheorem{cor}[thm]{\protect\corollaryname}
\theoremstyle{plain}
\newtheorem{lem}[thm]{\protect\lemmaname}
\theoremstyle{definition}
\newtheorem{example}[thm]{\protect\examplename}
\theoremstyle{plain}
\newtheorem{lyxalgorithm}[thm]{\protect\algorithmname}

\@ifundefined{date}{}{\date{}}
\usepackage{babel}
\usepackage{babel}
\usepackage{babel}
\usepackage{babel}
\usepackage{babel}
\usepackage{babel}
\usepackage{babel}
\usepackage{babel}
\usepackage{babel}
\usepackage{babel}
\usepackage{xpatch}
\usepackage{babel}
\usepackage{babel}
\usepackage{babel}

\usepackage[misc]{ifsym}

\usepackage{color}
\def\mystrut(#1,#2){\vrule height #1pt depth #2pt width 0pt}
\usepackage[normalem]{ulem}
\newcommand{\culine}{\bgroup\markoverwith
{\textcolor{red}{\rule[-0.5ex]{2pt}{0.5pt}}}\ULon}

\def\cdashuline{\bgroup
\UL@setULdepth
\markoverwith{\textcolor{red}{\kern.13em
\vtop{\kern\ULdepth \hrule width .3em}%
\kern.13em}}\ULon}
\def\cuuline{\bgroup \UL@setULdepth
\markoverwith{\textcolor{red}{\lower\ULdepth\hbox
{\kern-.03em\vbox{\hrule width.2em\kern1.2\p@\hrule}\kern-.03em}}}
\ULon}
\usepackage{enumitem}
\setlist[itemize,1]{label={\fontfamily{cmr}\fontencoding{T1}\selectfont\textbullet}}

\renewenvironment{proof}[1][\proofname] {\par\pushQED{\qed}\normalfont\topsep6\p@\@plus6\p@\relax\trivlist\item[\hskip\labelsep\bfseries#1\@addpunct{.}]\ignorespaces}{\popQED\endtrivlist\@endpefalse}

\AtBeginDocument{\xpatchcmd{\@thm}{\thm@headpunct{.}}{\thm@headpunct{.}}{}{}}

\usepackage{mdwlist}

\providecommand{\definitionname}{Definition}
\providecommand{\examplename}{Example}
\providecommand{\lemmaname}{Lemma}
\providecommand{\propositionname}{Proposition}
\providecommand{\remarkname}{Remark}
\providecommand{\theoremname}{Theorem}

\providecommand{\definitionname}{Definition}
\providecommand{\examplename}{Example}
\providecommand{\lemmaname}{Lemma}
\providecommand{\propositionname}{Proposition}
\providecommand{\remarkname}{Remark}
\providecommand{\theoremname}{Theorem}

\providecommand{\definitionname}{Definition}
\providecommand{\examplename}{Example}
\providecommand{\lemmaname}{Lemma}
\providecommand{\propositionname}{Proposition}
\providecommand{\remarkname}{Remark}
\providecommand{\theoremname}{Theorem}

\providecommand{\definitionname}{Definition}
\providecommand{\examplename}{Example}
\providecommand{\lemmaname}{Lemma}
\providecommand{\propositionname}{Proposition}
\providecommand{\remarkname}{Remark}
\providecommand{\theoremname}{Theorem}

\providecommand{\corollaryname}{Corollary}
\providecommand{\definitionname}{Definition}
\providecommand{\examplename}{Example}
\providecommand{\lemmaname}{Lemma}
\providecommand{\propositionname}{Proposition}
\providecommand{\remarkname}{Remark}
\providecommand{\theoremname}{Theorem}

\providecommand{\lemmaname}{Lemma}
\providecommand{\propositionname}{Proposition}
\providecommand{\remarkname}{Remark}
\providecommand{\theoremname}{Theorem}

\providecommand{\lemmaname}{Lemma}
\providecommand{\propositionname}{Proposition}
\providecommand{\remarkname}{Remark}
\providecommand{\theoremname}{Theorem}

\providecommand{\lemmaname}{Lemma}
\providecommand{\propositionname}{Proposition}
\providecommand{\remarkname}{Remark}
\providecommand{\theoremname}{Theorem}

\providecommand{\lemmaname}{Lemma}
\providecommand{\remarkname}{Remark}
\providecommand{\theoremname}{Theorem}

\providecommand{\lemmaname}{Lemma}
\providecommand{\remarkname}{Remark}
\providecommand{\theoremname}{Theorem}

\providecommand{\lemmaname}{Lemma}
\providecommand{\remarkname}{Remark}
\providecommand{\theoremname}{Theorem}

\providecommand{\corollaryname}{Corollary}
\providecommand{\definitionname}{Definition}
\providecommand{\examplename}{Example}
\providecommand{\lemmaname}{Lemma}
\providecommand{\remarkname}{Remark}
\providecommand{\theoremname}{Theorem}

\providecommand{\corollaryname}{Corollary}
\providecommand{\definitionname}{Definition}
\providecommand{\examplename}{Example}
\providecommand{\lemmaname}{Lemma}
\providecommand{\remarkname}{Remark}
\providecommand{\theoremname}{Theorem}
\renewcommand\theenumi{(\roman{enumi})}
\usepackage{microtype}
\usepackage{indentfirst}
\DeclareMathOperator{\im}{Im}

\makeatother

\providecommand{\algorithmname}{Algorithm}
\providecommand{\corollaryname}{Corollary}
\providecommand{\definitionname}{Definition}
\providecommand{\examplename}{Example}
\providecommand{\lemmaname}{Lemma}
\providecommand{\propositionname}{Proposition}
\providecommand{\remarkname}{Remark}
\providecommand{\theoremname}{Theorem}

\begin{document}
\title{\textbf{General Perturbation Resilient Dynamic String-Averaging for
Inconsistent Problems with Superiorization}}
\author{Kay Barshad and Yair Censor\\
Department of Mathematics, University of Haifa, Mt. Carmel, \\
Haifa 3498838, Israel \\
\Letter ~ \href{mailto:kaybarshad@gmail.com}{kaybarshad@gmail.com};
\Letter ~ \href{mailto:yair@math.haifa.ac.il}{yair@math.haifa.ac.il}}
\date{December 23, 2024. Revised: June 15, 2025.}
\maketitle
\begin{abstract}
\noindent In this paper we introduce a General Dynamic String-Averaging
(GDSA) iterative scheme and investigate its convergence properties
in the inconsistent case, that is, when the input operators don't
have a common fixed point. The Dynamic String-Averaging Projection
(DSAP) algorithm itself was introduced in an 2013 paper, where its
strong convergence and bounded perturbation resilience were studied
in the consistent case (that is, when the sets under consideration
had a nonempty intersection). Results involving combination of the
DSAP method with superiorization, were presented in 2015. The proof
of the weak convergence of our GDSA method is based on the notion
of ``strong coherence'' of sequences of operators that was introduced
in 2019. This is an improvement of the property of ``coherence''
of sequences of operators introduced in 2001 by Bauschke and Combetts.
Strong coherence provides a more convenient sufficient convergence
condition for methods that employ infinite sequences of operators
and it turns out to be a useful general tool when applied to proving
the convergence of many iterative methods. In this paper we combine
the ideas of both dynamic string-averaging and strong coherence, in
order to analyze our GDSA method for a general class of operators
and its bounded perturbation resilience in the inconsistent case with
weak and strong convergence. We then discuss an application of the
GDSA method to the Superiorization Methodology, developing results
on the behavior of its superiorized version.\\

\noindent Communicated by Fabian Flores-Bazàn\\


\end{abstract}

\section{Introduction}

A strongly convergent Dynamic String-Averaging Projection (DSAP) algorithm,
based on convex combinations and compositions of operators, was introduced,
along with its bounded perturbation resilience, by Censor and Zaslavski
in \cite{DSAP} for a family of nonempty, closed and convex sets $\left\{ C_{i}\right\} _{i=1}^{m}$,
where the considered operators were given by a sequence of metric
projections $\left\{ P_{C_{i}}\right\} _{i=1}^{m}$ in the consistent
case, that is, when $\cap_{i=1}^{m}C_{i}\not=\emptyset$. A superiorized
version of this algorithm appeared in \cite{CZ_sup}, where it was
proved that any sequence generated by it either converges to a constrained
minimum point of the objective function employed in the superiorization
process, or that it is strictly Fej{\'e}r monotone with respect to
a subset of the solution set of the constrained minimization problem.
String-averaging projection (SAP) methods form a general algorithmic
framework introduced in \cite{CEH2001}. Subsequently, they were developed
in a variety of situations such as for convex feasibility with infinitely
many sets \cite{Kong2019}, for incremental stochastic subgradient
algorithms \cite{costa} and for proton computed tomography image
reconstruction \cite{Barsik}, to name but a few. See also \cite{Bargetz2018},
where perturbation resilience of such methods was further studied
and \cite{Bargetz_LC}, where linear convergence rates for a certain
class of extrapolated fixed point algorithms which are based on dynamic
string-averaging methods were established.

In the present paper we consider more general ideas, focusing on more
general algorithmic structures which are applied to the inconsistent
case (that is, when the input operators don't have a common fixed
point), and when the objective function, involved in the superiorization,
is only required to be convex and continuous. For a positive integer
$m$ and a given family of nonexpansive operators $\left\{ U_{i}\right\} _{i=1}^{m}$,
$U_{i}:\mathcal{H}\rightarrow\mathcal{H}$ for each $i=1,2,\dots,m$,
without a common fixed point, we construct a General Dynamic String-Averaging
(GDSA) algorithm (Algorithm \ref{alg:GDSA} below) based on convex
combinations, compositions and relaxations of the operators of the
aforementioned family. In Theorem \ref{main_res} below we investigate
weak and strong convergence properties of our GDSA algorithm and its
bounded perturbation resilience. To this end we introduce a ``$\limsup$-admissible
sequence of operators'' (Definition \ref{def:admis} below) and show
that if, after a certain GDSA procedure, the sequence of output operators
is $\limsup$-admissible, then the algorithm converges weakly (and
under additional assumptions -- strongly) to a point in a certain
(assumed nonempty) set, even though the given input operators don't
have a common fixed point.

We further apply our GDSA algorithm to the Superiorization Methodology
(SM). The SM is not aiming to solve the constrained minimization problem
under consideration, but to find a feasible point for the original
feasibility-seeking problem which is ``superior'', i.e., has smaller
or equal objective function value than that of a point returned by
the feasibility-seeking only algorithm that is employed by the SM.
More details about the SM appear below in Subsection 5.1.

The weak convergence of our GDSA algorithm and its superiorized version
is based, inter alia, on the theory of coherence, which was introduced
by Bauschke and Combetts in \cite{BC} and further developed by Barshad,
Reich and Zalas in \cite{BRZ(R)}, see also \cite{Bar_thes} and \cite{GMSA}.

All the results below, concerning our algorithmic schemes, remain
valid in the consistent case as well. However, since it is generally
desirable to assume certain properties of the input operators, which
is possible in this case as it was shown in \cite{MSA}, such results
are of limited interest in the consistent case due to the assumptions
placed on the output operators in our GDSA procedure.

It is worth noting the significance of the inconsistent case at this
point. The convex feasibility problem (CFP) is to find a feasible
point in the intersection of a family of convex and closed sets. If
the intersection is empty, then the CFP is inconsistent and a feasible
point does not exist. However, algorithmic research on inconsistent
CFPs does exist and is mainly focused on two directions. One is oriented
toward defining other solution concepts that will apply, such as proximity
function minimization, wherein a proximity function measures, in some
way, the total violation of all constraints (see, for instance, Example
\ref{SimPro} below). The second direction investigates the behavior
of algorithms that are designed to solve a consistent CFP when applied
to inconsistent problems. This direction is fueled by situations wherein
one lacks a priori information about the consistency or inconsistency
of the CFP or does not wish to invest computational resources to get
hold of such knowledge prior to running his algorithm. A telegraphical
review of some recent works on inconsistent CFPs appears in \cite{CensorZaknoon2018}.

The rest of the paper is organized as follows. In Section \ref{sec2}
we provide some background which is needed to establish our results.
In Section \ref{sec3} we introduce the general bounded regularity
with its properties and discuss the well-known notion of approximate
shrinking of operators. In Section \ref{sec4} we present the properties
of our GDSA algorithm. Notations and initial tools appear in Subsection
\ref{subsec:Notions,-notations-and} and lemmata leading to the main
result are in Subsection \ref{subsec:The-convergent-and}. In Section
\ref{sec5}, we present applications of this algorithm to the Superiorization
Methodology, followed by a brief summary with conclusions in Section
\ref{sec:Conclusion}.

\section{\label{sec2}Preliminaries }

Throughout this paper $\mathbb{N}$ denotes the set of natural numbers
(starting from $0$), and for any two integers $m$ and $n$, with
$m\le n$, we denote by $\left\{ m,m+1,\dots,n\right\} $ the set
of all integers between $m$ and $n$. The real line is denoted by
$\mathbb{R}$. For a set $A$, we denote by $\left|A\right|$ the
cardinality of $A$. For a real Hilbert space $\mathcal{H}$, we use
the following notations:
\begin{itemize}
\item $\langle\cdot,\cdot\rangle$ denotes the inner product on $\mathcal{H}.$
\item $\Vert\cdot\Vert$ denotes the norm on $\mathcal{H}$ induced by $\langle\cdot,\cdot\rangle.$
\item $Id$ denotes the identity operator on $\mathcal{H}$.
\item $\mathrm{Fix}T$ denotes the, possibly empty, set $\mathrm{Fix}T:=\{x\in\mathcal{H}\mid T(x)=x\}$
of fixed points of an operator $T:\mathcal{H}\rightarrow\mathcal{H}$.
\item For a nonempty and convex subset $C$ of $\mathcal{H}$, we denote
by $P_{C}$ the (unique) metric projection onto $C$, the existence
of which is guaranteed if $C$ is, in addition, closed.
\item The expressions $x^{k}\rightharpoonup x$ and $x^{k}\rightarrow x$
denote, the weak and strong convergence, respectively, to $x$ of
a sequence $\left\{ x^{k}\right\} _{k=0}^{\infty}$ in $\mathscr{\left(\mathcal{H}\mathscr{,\Vert\cdot\Vert}\right)}$when
$k\rightarrow\infty$, while $\mathfrak{W}\left(\left\{ \mathit{x^{k}}\right\} _{k=0}^{\infty}\right)$
denotes the set of weak cluster points of $\left\{ x^{k}\right\} _{k=0}^{\infty}$.
\item For a convex function $\phi:\mathcal{H}\rightarrow\mathbb{R}$ and
a point $x\in\mathcal{H}$, we denote by $\partial\phi\left(x\right)$
the subdifferential of $\phi$ at $x$, that is,
\begin{equation}
\partial\phi\left(x\right):=\left\{ g\in\mathcal{H}\,|\,\left\langle g,y-x\right\rangle \le f\left(y\right)-f\left(x\right)\,\,\mathrm{for}\,\,\mathrm{all}\,\,y\in\mathcal{H}\right\} .\label{eq:-24}
\end{equation}
\item For a function $f:\mathcal{H}\rightarrow\mathbb{R}$ and a subset
$A$ of $\mathcal{H}$, we denote by $\underset{x\in A}{\mathrm{Argmin}}f\left(x\right)$
the set of minimizers of $f$ on the set $A$.
\item $B\left(x,r\right)$ denotes the open ball centered at $x\in\mathcal{H}$
of radius $r>0$.
\item For a nonempty subset $C$ of $\mathcal{H}$ and $x\in\mathcal{H}$,
we denote by $d\left(x,C\right)$ the distance from $x$ to $C$,
that is, $d\left(x,C\right):=\inf_{y\in C}\left\Vert x-y\right\Vert $.
\end{itemize}
We recall the following types of algorithmic operators. For more information
on such operators, see, for example, \cite{C_book}.
\begin{defn}
Let $T:\mathcal{H}\rightarrow\mathscr{\mathcal{H}}$ be an operator
and let $\lambda\in\left[0,2\right]$. The operator $T_{\lambda}:\mathcal{H}\rightarrow\mathscr{\mathcal{H}}$
defined by $T_{\lambda}:=\left(1-\lambda\right)\mathrm{Id}+\lambda T$
is called a $\lambda$-$relaxation$ of the operator $T$. The operator
$T_{2}$ is called a \textit{reflection} of the operator $T$.
\end{defn}
\begin{defn}
An operator $T:\mathcal{H}\rightarrow\mathcal{H}$ is said to be \textit{nonexpansive}
(NE) if 
\[
\left(\forall x,y\in\mathcal{H}\right)\,\,\left\Vert T\left(x\right)-T\left(y\right)\right\Vert \leq\left\Vert x-y\right\Vert .
\]
For $\lambda\in\left[0,2\right]$, an operator $T:\mathcal{H}\rightarrow\mathcal{H}$
is said to be $\lambda$-\textit{relaxed nonexpansive} if $T$ is
a $\lambda$-relaxation of a nonexpansive operator $U$, that is,
$T=U_{\lambda}$.
\end{defn}
\begin{rem}
\label{conv_com_comp}Clearly, convex combinations and compositions
of nonexpansive operators are also nonexpansive. Moreover, for each
$\lambda\in\left[0,1\right]$, the $\lambda$-relaxation of a nonexpansive
operator is also nonexpansive.
\end{rem}
A central role in our analysis is played by cutter operators that
we define next. For every ordered pair $\left(x,y\right)\in\mathscr{\mathcal{H}^{\mathit{2}}}$,
we define the closed and convex set $H\left(x,y\right)$ by

\[
H\left(x,y\right):=\left\{ u\in\mathcal{H}\,\vert\,\langle u-y,x-y\rangle\le0\right\} .
\]

The following class of cutters was introduced by Bauschke and Combettes
in \cite{BC}, with a different terminology and named there the ``class
$\mathfrak{\mathfrak{T}}$''. The name ``cutter'' was proposed
in \cite{cegielski-censor-2011}; other names are used in the literature
for these operators, such as, e.g., ``firmly quasi-nonexpansive''
(see, for example, Definition 4.1 and Proposition 4.2 in \cite{BC_book}),
where various properties and examples of cutters can be found.
\begin{defn}
An operator $T:\mathscr{\mathcal{H}}\rightarrow\mathscr{\mathcal{H}}$
is called a \textit{cutter} if it satisfies the condition 
\[
\textnormal{Fix}T\subseteq H\left(x,T\left(x\right)\right),\text{ }\forall x,y\in X,
\]
or, equivalently, $\left\langle z-T(x),x-T(x)\right\rangle \leq0$
for each $z\in\mathrm{Fix}T$ and $x\in X$. For $\lambda\in\left[0,2\right]$,
an operator $T:\mathcal{H}\rightarrow\mathcal{H}$ is a $\lambda$-\textit{relaxed
}cutter if $T$ is a $\lambda$-relaxation of a cutter $U$, that
is, $T=U_{\lambda}=\left(1-\lambda\right)\mathrm{Id}+\lambda U$.
\end{defn}
\begin{prop}
Let $T$ be a cutter. Then $\mathrm{Fix}T=\cap_{x\in\mathcal{H}}H\left(x,T\left(x\right)\right)$
and hence $\mathrm{Fix}T$ is a closed and convex subset of $\mathcal{H}$,
as an intersection of half-spaces.
\end{prop}
\begin{proof}
See Proposition 2.6 in \cite{BC}.
\end{proof}
\begin{rem}
\label{Rem_cut}It is worthwhile to caution the reader about some
ambiguity in the literature regarding the term cutter. Definition
9.2 of \cite{cegielski-censor-2011} is the first original definition
of cutters and they are defined there without any condition on the
nonemptiness of the fixed points set of the operators. Cutters are,
thus, just another name for the members of the original ``class $\mathfrak{\mathfrak{T}}$''
of Bauschke and Combettes in \cite{BC} which are also defined there
without any condition on the nonemptiness of the fixed points set
of the operators. However, in Definition 2.1.30 in \cite{C_book}
the condition that the fixed points set of the cutter operators should
be nonempty was included in the definition. This created some ambiguity
due to the fact that some later publications either include or do
not include the condition that the fixed points set of the cutter
operators should be nonempty.

We adhere to the original definition (Definition 9.2 of \cite{cegielski-censor-2011})
and when we need the fixed points set of the cutter to be nonempty
we explicitly say so.
\end{rem}
\begin{rem}
\label{Relaxation has the same Fix} Clearly, $\mathrm{Fix}T=\mathrm{Fix}T_{\lambda}$
for every operator $T:\mathcal{H}\rightarrow\mathcal{H}$ and every
$\lambda\in\left(0,2\right]$. Moreover, if $T$ is a cutter, then
for every $\lambda\in\left[0,1\right],$ $T_{\lambda}$ is also a
cutter since $H\left(x,Tx\right)\subset H\left(x,T_{\lambda}\left(x\right)\right)$
for each $x\in\mathcal{H}$ and each $\lambda\in\left[0,1\right]$.
\end{rem}
\begin{defn}
We say that an operator $T:\mathscr{\mathcal{H}}\rightarrow\mathcal{H}$
is:
\begin{enumerate}
\item \textit{Firmly nonexpansive }(FNE) if
\[
\left(\forall x,y\in\mathcal{H}\right)\left\langle T\left(x\right)-T\left(y\right),x-y\right\rangle \ge\left\Vert T\left(x\right)-T\left(y\right)\right\Vert ^{2}.
\]
\item $\rho$-\textit{firmly nonexpansive} ($\rho$-FNE), where $\rho\ge0$
is a real number,\textit{ }if 
\[
\left(\forall x,y\in\mathcal{H}\right)\left\Vert T\left(x\right)-T\left(y\right)\right\Vert ^{2}\le\left\Vert x-y\right\Vert ^{2}-\rho\left\Vert \left(x-T\left(x\right)\right)-\left(y-T\left(y\right)\right)\right\Vert ^{2}.
\]
\end{enumerate}
\end{defn}
\begin{defn}
For $\lambda\in\left[0,2\right]$, an operator $T:\mathcal{H}\rightarrow\mathcal{H}$
is called $\lambda$-\textit{relaxed firmly nonexpansive} if $T$
is a $\lambda$-relaxation of a firmly nonexpansive operator $U$,
that is, $T=U_{\lambda}=\left(1-\lambda\right)\mathrm{Id}+\lambda U$.
\end{defn}
\begin{thm}
\label{thm:2.2.5}If $T:\mathscr{\mathcal{H}}\rightarrow\mathcal{H}$
is firmly nonexpansive, then $T$ is a nonexpansive cutter.
\end{thm}
\begin{proof}
See Theorem 2.2.4 and Theorem 2.2.5 in \cite{C_book}.
\end{proof}
\begin{thm}
\label{FNE-equiv.cond}Let $T:\mathcal{H}\rightarrow\mathcal{H}$
be an operator. The following conditions are equivalent:
\begin{enumerate}
\item $T$ is firmly nonexpansive.
\item $T_{\lambda}$ is nonexpansive for each $\lambda\in\left[0,2\right]$.
\item There exists a nonexpansive operator $N:\mathcal{H\rightarrow\mathcal{H}}$
such that $T=2^{-1}\left(Id+N\right)$.
\end{enumerate}
\end{thm}
\begin{proof}
See Theorem 2.2.10 in \cite{C_book}.
\end{proof}
\begin{thm}
\label{FNE_relax_FNE}For any $\alpha\in\left(0,2\right]$, an operator
$T:\mathcal{H}\rightarrow\mathcal{H}$ is firmly nonexpansive if and
only if its relaxation $T_{\alpha}$ is $\left(2-\alpha\right)\alpha^{-1}$-firmly
nonexpansive, that is, if and only if
\[
\left(\forall x,y\in\mathcal{H}\right)\left\Vert T_{\alpha}\left(x\right)-T_{\alpha}\left(y\right)\right\Vert ^{2}\le\left\Vert x-y\right\Vert ^{2}-\left(2-\alpha\right)\alpha^{-1}\left\Vert \left(x-T_{\alpha}\left(x\right)\right)-\left(y-T_{\alpha}\left(y\right)\right)\right\Vert ^{2}.
\]
\end{thm}
\begin{proof}
See Corollary 2.2.15 in \cite{C_book}.
\end{proof}
\begin{cor}
\label{cor: frm-nexp}Let $U:\mathcal{H}\rightarrow\mathcal{H}$ be
an operator and let $T:=Id+2^{-1}\left(1+\rho\right)\left(U-Id\right)$
for some $\rho\ge0$, that is, $T=U_{2^{-1}\left(1+\rho\right)}$.
Then $U$ is $\rho$-firmly nonexpansive if and only if $T$ is firmly
nonexpansive. In particular, $U$ is nonexpansive if and only if $T:=\frac{1}{2}\left(U+Id\right)$
is firmly nonexpansive.
\end{cor}
\begin{proof}
Immediate from Theorem \ref{FNE_relax_FNE}.
\end{proof}
\begin{defn}
We say that an operator $T:\mathcal{H}\rightarrow\mathcal{H}$ is:
\begin{enumerate}
\item \textit{Quasi-nonexpansive} if 
\[
\left\Vert T\left(x\right)-z\right\Vert \le\left\Vert x-z\right\Vert 
\]
for each $x\in\mathcal{H}$ and $z\in\mathrm{Fix}T$.
\item \textit{$\rho$-strongly quasi-nonexpansive} for some $0\le\rho\in\mathbb{R}$
if
\begin{equation}
\left\Vert T\left(x\right)-z\right\Vert ^{2}\le\left\Vert x-z\right\Vert ^{2}-\rho\left\Vert T\left(x\right)-x\right\Vert ^{2}\label{eq:-1-1}
\end{equation}
for all $x\in\mathcal{H}$ and for all $z\in\mathrm{Fix}T$. If $T$
satisfies \eqref{eq:-1-1} for some $\rho>0$, then it is called strongly
quasi-nonexpansive.
\end{enumerate}
\end{defn}
\begin{rem}
\label{FNE-SQNE}Clearly, for any $\rho\ge0$, a $\rho$-strongly
quasi-nonexpansive operator is, in particular, quasi-nonexpansive
and a $\rho$-firmly nonexpansive operator operator is, in particular,
$\rho$-strongly quasi-nonexpansive.
\end{rem}
\begin{rem}
Similarly to Remark \ref{Rem_cut} about cutters, we mention that
in the literature (as, for example, Definitions 2.1.19 and 2.1.38
in \cite{C_book}) quasi-nonexpansive and strongly quasi-nonexpansive
operators are required to have a nonempty fixed point set by definition.
Here we define these operators without this requirement and assume
the nonemptiness of their fixed points set only if we need it.
\end{rem}
\begin{cor}
\label{comp_conv_comb_firm}For a positive integer $m$, let $\left\{ U_{i}\right\} _{i=1}^{m}$
be a finite family of $\rho_{i}$-firmly nonexpansive operators, where
$U_{i}:\mathcal{H\rightarrow\mathcal{H}}$ and $\rho_{i}\in\left[0,\infty\right)$
for each $i=1,2,\dots,m$. Set $\rho:=\min_{i\in\left\{ 1,2,\dots,m\right\} }\rho_{i}$.
Then:
\begin{enumerate}
\item For each finite family of numbers $\left\{ \omega_{i}\right\} _{i=1}^{m}\subset\left[0,1\right]$
such that $\sum_{i=1}^{m}\omega_{i}=1$, the convex combination $U:=\sum_{i=1}^{m}\omega_{i}U_{i}$
is $\rho$- firmly nonexpansive.
\item The composition $V:=U_{m}\cdots U_{2}U_{1}$ is $\rho m^{-1}$- firmly
nonexpansive.
\end{enumerate}
\end{cor}
\begin{proof}
See Theorems 2.2.35 and 2.2.42 in \cite{C_book} along with Theorem
\ref{FNE_relax_FNE} above.
\end{proof}
\begin{defn}
An operator $T:\mathcal{H\rightarrow\mathcal{H}}$ is \textit{weakly
regular} (i.e., satisfies \textit{Opial's demi-closedness principle
which says that $T-Id$ is demiclosed at $0$)} if for any sequence
$\left\{ x^{k}\right\} _{k=0}^{\infty}\subset\mathcal{H}$ and any
$x\in\mathcal{H}$, the following implication holds:
\[
\left\{ \begin{array}{c}
x^{k}\rightharpoonup x\\
T\left(x^{k}\right)-x^{k}\rightarrow0
\end{array}\Longrightarrow x\in\mathrm{Fix}T\right..
\]
\end{defn}
\begin{lem}
\label{lemma:3.2.5}Let $T:\mathcal{H}\rightarrow\mathcal{H}$ be
nonexpansive. Then $T$ is weakly regular.
\end{lem}
\begin{proof}
See Lemma 3.2.5 in \cite{C_book}.
\end{proof}
\begin{example}
\label{metric projection}Given a nonempty, closed an convex subset
$C$ of $\mathcal{H}$, the metric projection $P_{C}$ onto $C$ is
firmly nonexpansive (see, e.g., Theorem 2.2.21 in \cite{C_book})
and, hence, a nonexpansive cutter (by Theorem \ref{thm:2.2.5}) and
weakly regular (by Lemma \ref{lemma:3.2.5}). Moreover $\mathrm{Fix}P_{C}=C$.
\end{example}
\begin{defn}
For a nonempty, closed and convex subset $C$ of $\mathcal{H}$, the
sequence $\left\{ x^{k}\right\} _{k=0}^{\infty}$ in $\mathcal{H}$
is
\begin{enumerate}
\item \textit{Fej{\'e}r monotone} w\textit{ith respect to $C$} if for
each $z\in C$ and each $k\in\mathbb{N}$, 
\[
\left\Vert x^{k+1}-z\right\Vert \le\left\Vert x^{k}-z\right\Vert .
\]
\item \textit{Strictly Fej{\'e}r monotone} \textit{with respect to $C$}
if for each $z\in C$ and each $k\in\mathbb{N}$,
\[
\left\Vert x^{k+1}-z\right\Vert <\left\Vert x^{k}-z\right\Vert .
\]
\end{enumerate}
\end{defn}
\begin{thm}
\label{Fejer}Let $C$ be a nonempty, closed and convex subset of
$\mathcal{H}$. If $\left\{ x^{k}\right\} _{k=0}^{\infty}$ is Fej{\'e}r
monotone with respect to $C$, then it converges strongly to some
point in $C$ if and only if 
\[
\lim_{k\rightarrow\infty}d\left(x^{k},C\right)=0.
\]
\end{thm}
\begin{proof}
See Theorems 2.16(v) in \cite{BB96}.
\end{proof}
Next we investigate the properties of the following algorithm which
is the algorithmic framework for our GDSA method.
\begin{lyxalgorithm}
[The algorithmic framework]\label{alg 1}

Given $\varepsilon\in\left(0,1\right]$, $x^{0}\in\mathcal{H}$ and
a sequence $\left\{ T_{k}\right\} _{k=0}^{\infty}$ of operators,$T_{k}:\mathcal{H\rightarrow\mathcal{H}}$
for each $k\in\mathbb{N}$, let the algorithm be defined by the recurrence
\[
x^{k+1}:=x^{k}+\lambda_{k}\left(T_{k}\left(x^{k}\right)-x^{k}\right),
\]
where $\lambda_{k}\in\left[\varepsilon,2-\varepsilon\right]$ for
each $k\in\mathbb{N}$.
\end{lyxalgorithm}
The next definition of bounded perturbations resilience of an iterative
algorithm that is governed by an infinite sequence of algorithmic
operators generalizes the earlier commonly used definition, e.g.,
Definition 3.1 in \cite{DSAP}, wherein only a single operator was
considered. See \cite{BuRZ1}, where it was shown for the first time
that if the exact iterates of a nonexpansive operator converge, then
its inexact iterates with summable errors converge as well.
\begin{defn}
\label{def:resilient}\textbf{Bounded perturbations resilience}. Let
$\Gamma\subseteq\mathcal{H}$ be a given nonempty subset of $\mathcal{H}$
and $\left\{ T_{k}\right\} _{k=0}^{\infty}$ be a sequence of operators,
$T_{k}:\mathcal{H}\rightarrow\mathcal{H}$ for each $k\in\mathbb{N}$.
The algorithm $x^{k+1}:=T_{k}(x^{k}),$ for all $k\in\mathbb{N},$
is said to be \textit{weakly (strongly) bounded perturbations resilient}
with respect to\textbf{ $\Gamma$}\emph{ }if the following is true:
If a sequence $\{x^{k}\}_{k=0}^{\infty},$ generated by the algorithm,
converges weakly (strongly)\textit{ }to a point in $\Gamma$ for all
$x^{0}\in\mathcal{H}$, then any sequence $\{y^{k}\}_{k=0}^{\infty}$
in $\mathcal{H}$ that is generated by the algorithm $y^{k+1}:=T_{k}(y^{k}+\beta_{k}v^{k}),$
for all $k\in\mathbb{N},$ also converges weakly (strongly) to a point
in $\Gamma$ for all $y^{0}\in\mathcal{H},$ provided that $\{\beta_{k}v^{k}\}_{k=0}^{\infty}$
are bounded perturbations, meaning that $\left\{ \beta_{k}\right\} _{k=0}^{\infty}$
is a sequence of positive real numbers such that $\sum_{k=0}^{\infty}\beta_{k}<\infty$
and that the sequence $\{v^{k}\}_{k=0}^{\infty}$ is a bounded sequence
in $\mathcal{H}$.
\end{defn}
\begin{rem}
The terms \textit{``weakly (strongly)''} in the above definition
are related to the convergence of the considered sequences. They should
not be confused with the notions of weak and strong perturbation resilience
used in the theory of the superiorization method, see \cite{weak-strong-superiorization},
particularly Definition 9 therein.
\end{rem}
The notions of coherence and strong coherence presented next play
a fundamental role in our work here.
\begin{defn}
A sequence $\left\{ T_{k}\right\} _{k=0}^{\infty}$ of self-mapping
operators of $\mathcal{H}$ with the set of common fixed points $F:=\cap_{k=0}^{\infty}\mathrm{Fix}T_{k}$
is:
\begin{enumerate}
\item \textit{Coherent} if for every bounded sequence $\left\{ z^{k}\right\} _{k=0}^{\infty}$
in $\mathcal{H}$, we have
\[
\left\{ \begin{array}{c}
\sum_{k=0}^{\infty}\Vert z^{k+1}-z^{k}\Vert^{2}<\infty\\
\sum_{k=0}^{\infty}\Vert T_{k}\left(z^{k}\right)-z^{k}\Vert^{2}<\infty
\end{array}\Longrightarrow\mathfrak{W}\left(\left\{ \mathit{z^{k}}\right\} _{k=0}^{\infty}\right)\subset F.\right.
\]
\item \textit{Strongly coherent} if for any bounded sequence $\left\{ z^{k}\right\} _{k=0}^{\infty}$
in $\mathcal{H}$, we have
\[
\left\{ \begin{array}{c}
z^{k+1}-z^{k}\rightarrow0\\
T_{k}\left(z^{k}\right)-z^{k}\rightarrow0
\end{array}\Longrightarrow\mathfrak{W}\left(\left\{ \mathit{z^{k}}\right\} _{k=0}^{\infty}\right)\subset F.\right.
\]
\end{enumerate}
\end{defn}
\begin{rem}
\label{rem2.14}Clearly, a strongly coherent sequence of operators
is coherent. See Example 3.3 in \cite{BRZ(R)} to verify that this
inclusion is proper.
\end{rem}
\begin{lem}
\label{proposition 4.5 }Let $\delta\in\left(0,1\right]$ be a fixed
real number, $\left\{ \lambda_{k}\right\} _{k=0}^{\infty}$ be a real
sequence such that $\lambda_{k}\in\left[\delta,1\right]$ for each
$k\in\mathbb{N}$ and let $\left\{ T_{k}\right\} _{k=0}^{\infty}$
be a sequence of operators, $T_{k}:\mathcal{H}\rightarrow\mathcal{H}$
for each $k\in\mathbb{N}$. Then the sequence $\left\{ T_{k\lambda_{k}}\right\} _{k=0}^{\infty}$
of $\lambda_{k}$-relaxations of $\left\{ T_{k}\right\} _{k=0}^{\infty}$,
that is, $T_{k\lambda_{k}}:=Id+\lambda_{k}\left(T_{k}-Id\right)$
for each $k\in\mathbb{N}$, is coherent if and only if $\left\{ T_{k}\right\} _{k=0}^{\infty}$
is coherent.
\end{lem}
\begin{proof}
See Proposition 4.5 in \cite{BC}.
\end{proof}
\begin{rem}
\label{alg_variant} Algorithm \ref{alg 1} has a slightly different
formulation in \cite{BC}, where $\lambda_{k}:=2-\varepsilon$ for
each $k\in\mathbb{N}$. Due to Remark \ref{Relaxation has the same Fix}
and Lemma \ref{proposition 4.5 }, the following theorem (presented
in \cite{BC}) is valid in the case of our adjusted algorithm above
as well.
\end{rem}
\begin{thm}
\label{Main BC} Suppose that $\left\{ T_{k}\right\} _{k=0}^{\infty}$
is a coherent sequence of cutters. If $\cap_{k=0}^{\infty}\mathrm{Fix}T_{k}\not=\emptyset$,
then the sequence defined by Algorithm \ref{alg 1} converges weakly
to some $x\in\cap_{k=0}^{\infty}\mathrm{Fix}T_{k}$.
\end{thm}
\begin{proof}
See Theorem 4.2\textit{(i)} in \cite{BC}.
\end{proof}
The following theorem is a slight extension of Theorem 3.13 in \cite{BRZ(R)}.
We prove it here in the case of strong coherence for the convenience
of the reader.
\begin{thm}
\label{thm feas}Let $I$ be an arbitrary nonempty index set and let
$\left\{ C_{i}\right\} _{i\in I}$ be a family of (possibly empty)
closed and convex sets in $\mathcal{H}$ and let $\left\{ T_{k}\right\} _{k=0}^{\infty}$
be a sequence of operators, $S_{k}:\mathcal{H\rightarrow\mathcal{H}}$
for each $k\in\mathbb{N}$. Assume that 
\begin{equation}
C:=\cap_{i\in I}C_{i}\subset F:=\cap_{k=0}^{\infty}\mathrm{Fix}T_{k}\label{eq:-4}
\end{equation}
and $\left\{ I_{k}\right\} _{k=0}^{\infty}$ is an admissible control
sequence of subsets of $\mathbb{N}$, that is, for each $i\in I$,
there is an integer $M_{i}>0$ such that 
\begin{equation}
i\in\cup_{n=k}^{k+M_{i}-1}I_{n}\,\,\mathrm{for\,all}\,k\in\mathbb{N}.\label{eq:-12-1}
\end{equation}
Finally, suppose that for every $i\in I$, every $z\in\mathcal{H}$,
every bounded sequence $\left\{ z^{k}\right\} _{k=0}^{\infty}$ in
$\mathscr{\mathcal{H}}$, and every strictly increasing sequence $\left\{ n_{k}\right\} _{k=0}^{\infty}\subset\mathbb{N}$,
we have 
\begin{equation}
\left\{ \begin{array}{c}
z^{n_{k}}\rightharpoonup z\\
i\in I_{n_{k}}\,\,\mathrm{for\,all}\,k\in\mathbb{N}\\
z^{k+1}-z^{k}\rightarrow0\\
T_{k}\left(z^{k}\right)-z^{k}\rightarrow0
\end{array}\Longrightarrow z\in C_{i}.\right.\label{eq:-10}
\end{equation}
Then the sequence \textup{$\left\{ T_{k}\right\} _{k=0}^{\infty}$}
is strongly coherent.
\end{thm}
\begin{proof}
We first show that under the assumptions of the theorem we have for
each bounded sequence $\left\{ z^{k}\right\} _{k=0}^{\infty}$ in
$\mathcal{H}$, 
\begin{equation}
\left\{ \begin{array}{c}
z^{k+1}-z^{k}\rightarrow0\\
T_{k}\left(z^{k}\right)-z^{k}\rightarrow0
\end{array}\Longrightarrow\mathfrak{W}\left(\left\{ \mathit{z^{k}}\right\} _{k=0}^{\infty}\right)\subset C.\right.\label{eq:-2}
\end{equation}
Let $\left\{ z^{k}\right\} _{k=0}^{\infty}$ be a bounded sequence
in $\mathcal{H}$ satisfying
\begin{equation}
\left\{ \begin{array}{c}
z^{k+1}-z^{k}\rightarrow0\\
T_{k}\left(z^{k}\right)-z^{k}\rightarrow0
\end{array}\right.\label{eq:-3}
\end{equation}
and let $z\in\mathfrak{W}\left(\left\{ \mathit{z^{k}}\right\} _{k=0}^{\infty}\right)$.
Then there is a strictly increasing sequence $\left\{ n_{k}\right\} _{k=0}^{\infty}$
of natural numbers such that $z^{n_{k}}\rightharpoonup z$. Suppose
that $i\in I$. By \eqref{eq:-12-1}, there is an $M_{i}$ such that
the condition in \eqref{eq:-12-1} holds for all $k\in\mathbb{N}$.
Therefore, there is a sequence $\left\{ p_{k}\right\} _{k=0}^{\infty}$
in $\mathbb{N}$ such that
\[
\left(\forall k\in\mathbb{N}\right)n_{k}\le p_{k}\le n_{k}+M_{i}-1\,\,\,\mathrm{and}\,\,\,i\in I_{p_{k}}.
\]
Without any loss of generality, we may assume that $n_{k+1}-n_{k}\ge M_{i}$
for each $k\in\mathbb{N}$, because otherwise we can choose a subsequence
of $\left\{ z^{n_{k}}\right\} _{k=0}^{\infty}$ with this property.
Thus, we can assume that $\left\{ p_{k}\right\} _{k=0}^{\infty}$
is strictly increasing. Moreover, 
\begin{equation}
\left(\forall k\in\mathbb{N}\right)z^{p_{k}}=\sum_{j=0}^{p_{k}-1-n_{k}}\left(z^{n_{k}+j+1}-z^{n_{k}+j}\right)+z^{n_{k}}.\label{eq:-11}
\end{equation}
(By definition $\sum_{j=0}^{-1}\left(z^{n_{k}+j+1}-z^{n_{k}+j}\right):=0$).
By \eqref{eq:-3} and due to the finite number of summands in \eqref{eq:-11},
which is at most $M_{i}$, we obtain $z^{p_{k}}\rightharpoonup z$.
By the condition in \eqref{eq:-10} with respect to the sequence $\left\{ p_{k}\right\} _{k=0}^{\infty}$,
we have $z\in C_{i}$. This is true for each $i\in I$; hence, $z\in C$.
Thus, $\mathfrak{W}\left(\left\{ \mathit{z^{k}}\right\} _{k=0}^{\infty}\right)\subset C$
and \eqref{eq:-2} holds. Combining this with \eqref{eq:-4}, we see
that the sequence $\left\{ T_{k}\right\} _{k=0}^{\infty}$ is strongly
coherent. Theorem \ref{thm feas} is now proved.
\end{proof}
\begin{rem}
Under the assumptions of Theorem \ref{thm feas}, if all $\left\{ T_{k}\right\} _{k=0}^{\infty}$
are cutters then Theorem \ref{Main BC} holds true when applied to
$\left\{ T_{k}\right\} _{k=0}^{\infty}$ and $F$ is replaced by $C$.
See Theorem 2.2.1 and Remark 2.2.4 in \cite{Bar_thes} in this connection.
\end{rem}
\begin{thm}
\label{error}Let $C\subset\mathcal{H}$ be nonempty and closed. Let
\textup{$\left\{ T_{k}\right\} _{k=0}^{\infty}$}, $T_{k}:\mathscr{\mathcal{H}\rightarrow\mathcal{H}}$
for each $k\in\mathbb{N}$, be a sequence of nonexpansive operators
satisfying $C\subset\cap_{k=0}^{\infty}\mathrm{Fix}T_{k}$. Assume
that for each $y\in\mathcal{H}$ and each $q\in\mathbb{N}$, the sequence
$\left\{ T_{q+k}\cdots T_{q+1}T_{q}\left(y\right)\right\} _{k=0}^{\infty}$
converges weakly (strongly) to an element of $C$. Let $\left\{ \gamma_{k}\right\} _{k=0}^{\infty}\subset\left[0,\infty\right)$
be a sequence such that $\sum_{\text{k}=0}^{\infty}\gamma_{k}<\infty$
and let $\left\{ y^{k}\right\} _{k=0}^{\infty}\subset\mathcal{H}$.
Further assume that for each $k\in\mathbb{N}$,
\[
\left\Vert y^{k+1}-T_{k}\left(y^{k}\right)\right\Vert \le\gamma_{k}.
\]
Then the sequence $\left\{ y^{k}\right\} _{k=0}^{\infty}$converges
weakly (strongly) to an element of $C$.
\end{thm}
\begin{proof}
See Theorems 3.2 and 5.2 in \cite{BuRZ}.
\end{proof}
In the sequel we employ the following useful property of a convex
and continuous function.
\begin{thm}
\label{subdiff_ne+b}Let the function $\phi:\mathcal{H}\rightarrow\mathbb{R}$
be convex and continuous at the point $x\in\mathcal{H}$. Then the
subgradient set $\partial\phi\left(x\right)$ is nonempty.
\end{thm}
\begin{proof}
See Theorem 16.17(ii) in \cite{BC_book}.
\end{proof}

\section{\label{sec3}The general bounded regularity and approximately shrinking
operators}

In this section we consider the notions of bounded regularity and
approximate shrinking which are needed for establishing the strong
convergence of our algorithms.

The property of bounded regularity of a finite family of sets was
studied in \cite[Section 5]{BB96} and \cite{BB93}. Below we expand
its definition to hold for an arbitrary family of sets and show that
such a general boundedly regular family of sets possesses the same
properties as a finite one.
\begin{defn}
\label{bound_reg}For a nonempty index set $I$, the family $\left\{ C_{i}\right\} _{i\in I}$
of nonempty, closed and convex subsets of $\mathcal{H}$ with nonempty
intersection $C$ is \textit{boundedly regular }if for any bounded
sequence $\left\{ x^{k}\right\} _{k=0}^{\infty}$ in $\mathcal{H}$,
the following implication holds:

\textit{
\[
\lim_{k\rightarrow\infty}d\left(x^{k},C_{i}\right)=0\,\,\mathrm{for}\,\,\mathrm{each\,\,}i\in I\,\,\Longrightarrow\,\,\lim_{k\rightarrow\infty}d\left(x^{k},C\right)=0.
\]
}
\end{defn}
The next proposition provides sufficient conditions for bounded regularity.
We recall that a topological space $X$ is locally compact if each
$x\in X$ has a compact neighborhood with respect to the topology
inherited from $X$.
\begin{prop}
\label{bounded_reg}Let $\left\{ C_{i}\right\} _{i\in I}$ be a family
of nonempty, closed and convex subsets of $\mathcal{H}$ with a nonempty
intersection $C$. Then the following assertions hold:
\begin{enumerate}
\item If there is $i_{0}\in I$ for which the set $C_{i_{0}}$ is a locally
compact topological space (with respect to the norm topology inherited
from $\mathcal{H}$), then the family $\left\{ C_{i}\right\} _{i\in I}$
is boundedly regular.
\item If $\mathcal{H}$ is of finite dimension, then the family $\left\{ C_{i}\right\} _{i\in I}$
is boundedly regular.
\end{enumerate}
\end{prop}
\begin{proof}
\textit{(i)} Let $i_{0}\in I$ be an index for which the set $C_{i_{0}}$
is a locally compact topological space with respect to the norm topology
inherited from $\mathcal{H}$. Let $\left\{ x^{k}\right\} _{k=0}^{\infty}\subset\mathcal{H}$
be a bounded sequence such that 
\begin{equation}
\lim_{k\rightarrow\infty}d\left(x^{k},C_{i}\right)=0\,\,\mathrm{for}\,\,\mathrm{each\,\,}i\in I.\label{eq:-18}
\end{equation}
We need to show that $\lim_{k\rightarrow\infty}d\left(x^{k},C\right)=0.$
Assume to the contrary that this is not true. Then, since the sequence
$\left\{ x^{k}\right\} _{k=0}^{\infty}$ is bounded, we may assume,
without any loss of generality, that $\lim_{k\rightarrow\infty}d\left(x^{k},C\right)$
exists in $\mathbb{R}$,
\begin{equation}
\lim_{k\rightarrow\infty}d\left(x^{k},C\right)\not=0\label{eq:-22}
\end{equation}
and $\left\{ x^{k}\right\} _{k=0}^{\infty}$ converges weakly to some
point $x\in\mathcal{H}$, because otherwise there exists a subsequence
$\left\{ x^{n_{k}}\right\} _{k=0}^{\infty}$ of $\left\{ x^{k}\right\} _{k=0}^{\infty}$
with these properties. By \eqref{eq:-18}, we see that 
\begin{equation}
x^{k}-P_{C_{i_{0}}}\left(x^{k}\right)\rightarrow0.\label{eq:-19}
\end{equation}
It follows from \eqref{eq:-19} that the sequence $\left\{ P_{C_{i_{0}}}\left(x^{k}\right)\right\} _{k=0}^{\infty}\subset C_{i_{0}}$
also converges weakly to $x$ and hence is bounded. Since $C_{i_{0}}$
is closed and convex, its local compactness implies its bounded compactness
(that is, each closed ball in $C_{i_{0}}$, with respect to the norm
topology inherited from $\mathcal{H}$, is compact). Therefore, the
bounded sequence $\left\{ P_{C_{i_{0}}}\left(x^{k}\right)\right\} _{k=0}^{\infty}$
has a subsequence which converges strongly to $x$. Thus, we may assume
that $\left\{ P_{C_{i_{0}}}\left(x^{k}\right)\right\} _{k=0}^{\infty}$
converges strongly to $x$. By \eqref{eq:-19}, the sequence $\left\{ x^{k}\right\} _{k=0}^{\infty}$
also converges strongly to $x$. Now let $i\in I$ and let $\varepsilon>0$
be an arbitrary positive number. Then there exists a $k_{0}\in\mathbb{N}$
such that $\left\Vert x^{k}-x\right\Vert <\varepsilon$ for each natural
$k\ge k_{0}$, and hence, for each $c\in C_{i}$, we obtain, by the
triangle inequality,

\begin{equation}
d\left(x,C_{i}\right)\le\left\Vert x-c\right\Vert \le\left\Vert x-x^{k}\right\Vert +\left\Vert x^{k}-c\right\Vert <\varepsilon+\left\Vert x^{k}-c\right\Vert ,\label{eq:-20}
\end{equation}
for each natural $k\ge k_{0}$. Since \eqref{eq:-20} holds for an
arbitrary $c\in C_{i}$, we obtain, by \eqref{eq:-18}, that for each
natural $k\ge k_{0}$, 
\begin{equation}
d\left(x,C_{i}\right)<\varepsilon+d\left(x^{k},C_{i}\right)\rightarrow\varepsilon.\label{eq:-21}
\end{equation}
The arbitrariness of $\varepsilon$ along with \eqref{eq:-21} imply
that $d\left(x,C_{i}\right)=0$ and since the set $C_{i}$ is closed,
we have $x\in C_{i}$. It follows that $x\in C$ because $i\in I$
is arbitrary. By the continuity of the distance function, 
\[
\lim_{k\rightarrow\infty}d\left(x^{k},C\right)=d\left(x,C\right)=0
\]
which contradicts \eqref{eq:-22}.

\textit{(ii)} Since any closed subset of a finite-dimensional space
is locally compact, the result follows from \textit{(i)}.

This completes the proof of Proposition \ref{bounded_reg}.
\end{proof}
\begin{rem}
Note that the local compactness of the set $C_{i_{0}}$ in Proposition
\ref{bounded_reg}\textit{(i)} can be replaced by the bounded compactness,
since these properties are equivalent for a closed and convex subset
of a normed linear space.
\end{rem}
In the next example we show that the above conditions are sufficient,
but not necessary, for bounded regularity, as well as present an infinite
family of nonempty, closed and convex subsets which is not boundedly
regular. For an example of such a finite family, see, for instance,
Example 5.5 in \cite{BB93}.
\begin{example}
Set $\mathcal{H}:=l_{2}$. Let $\left\{ e^{k}\right\} _{k=0}^{\infty}$
be the sequence of elements in $l_{2}$, defined by $e_{n}^{k}:=\begin{cases}
1, & \mathrm{if}\,\,n=k,\\
0, & \mathrm{otherwise,}
\end{cases}$ for each $n\in\mathbb{N}$. Let $C_{n}:=\overline{\mathrm{sp}\left\{ e^{k}\right\} _{k=n}^{\infty}}$
for each $n\in\mathbb{N}$, where $\mathrm{sp}$ denotes the span
of a set of vectors and the upper bar denotes the closure. Clearly,
$\cap_{n\in\mathbb{N}}C_{n}:$$=\left\{ 0\right\} \not=\emptyset$
and for each $n\in\mathbb{N}$, we have $e^{k}\in C_{n}$ for each
natural $k\ge n$. This implies that $\lim_{k\rightarrow\infty}d\left(e^{k},C_{i}\right)=0$
for each $n\in\mathbb{N}$. However, $\lim_{k\rightarrow\infty}d\left(e^{k},C\right)=1\not=0$
and, therefore, the family $\left\{ C_{n}\right\} _{n\in\mathbb{N}}$
is not boundedly regular.

Observe that if we choose $C_{n}:=\mathrm{sp}\left\{ e^{n}\right\} $
for each $n\in\mathbb{\mathbb{N}}$ in the setting of this example,
then by Proposition \ref{bounded_reg}, the family $\left\{ C_{n}\right\} _{n\in\mathbb{N}}$
is a boundedly regular family of nonempty, closed and convex subsets
of $\mathcal{H}$ with the nonempty intersection $\cap_{n\in\mathbb{N}}C_{n}:$$=\left\{ 0\right\} \not=\emptyset$,
since each $C_{n}$ is a finite-dimensional normed linear space and
hence is locally compact.

Note that the local compactness in Proposition \ref{bounded_reg}
is a sufficient condition for the bounded regularity, but not a necessary
one. For instance, if we choose in the settings of this example $C_{n}:=\overline{B\left(0,\left(n+1\right)^{-1}\right)}$
for each $n\in\mathbb{N}$, then $\left\{ C_{n}\right\} _{n\in\mathbb{N}}$
is a regularly bounded family of nonempty closed and convex subsets
of $\mathcal{H}$ with the nonempty intersection $\cap_{n\in\mathbb{N}}C_{n}=\left\{ 0\right\} $,
but for each $n\in\mathbb{N}$, $C_{n}$ is not a locally compact
subspace of $\mathcal{H}$.
\end{example}
We also recall the following notion of approximate shrinking which
was extensively studied in \cite{CiegZ_app_s}.
\begin{defn}
A quasi-nonexpansive operator $T:\mathcal{H}\rightarrow\mathcal{H}$
\textit{is approximately shrinking} if for each bounded sequence $\left\{ x^{k}\right\} _{k=0}^{\infty}$
in $\mathcal{H}$, the following implication holds:
\[
\lim_{k\rightarrow\infty}\left\Vert T\left(x^{k}\right)-x^{k}\right\Vert =0\,\,\Longrightarrow\,\,\lim_{k\rightarrow\infty}d\left(x^{k},\mathrm{Fix}T\right)=0.
\]
\end{defn}
\begin{example}
\label{ex_as}Given a nonempty, closed and convex subset $C$ of $\mathcal{H}$,
the metric projection $P_{C}$ onto $C$ is approximately shrinking
(see Example 3.5 in \cite{CiegZ_app_s}). However, the quasi-nonexpansive
operator $U:\mathcal{H\rightarrow\mathcal{H}}$ defined by $U\left(x\right):=\begin{cases}
P_{\overline{B\left(0,2\right)}}\left(x\right), & \mathrm{if}\,\,x\not\in\overline{B\left(0,2\right),}\\
P_{\overline{B\left(0,1\right)}}\left(x\right), & \mathrm{if}\,\,x\in\overline{B\left(0,2\right),}
\end{cases}$ for each $x\in\mathcal{H}$ is not an approximately shrinking one
(see Example 3.7 on p. 404 in \cite{CiegZ_app_s}).
\end{example}

\section{\label{sec4}The convergence and bounded perturbation resilience
of the GDSA method}

In this section we consider a sequence $\left\{ T_{\left(\varOmega_{k},\omega_{k}\right)}\right\} _{k=0}^{\infty}$,
defined below, of nonexpansive operators, under the assumption of
an admissible control, with respect to which Algorithm \ref{alg 1}
converges to a point in a certain set and is bounded perturbations
resilient.

\subsection{\label{not}Notions, notations and initial tools\label{subsec:Notions,-notations-and}}

Throughout the rest of the paper we refer, among other things, to
the following setting defined below. We recall that for an arbitrary
sequence of sets $\left\{ A_{k}\right\} _{k=0}^{\infty}$, $\limsup_{k\rightarrow\infty}A_{k}:=\cap_{n=0}^{\infty}\cup_{k=n}^{\infty}A_{k}$.

Let $m$ be a positive integer. We consider a finite family $\left\{ U_{i}\right\} _{i=1}^{m}$
of $\alpha_{i}$-relaxed firmly nonexpansive operators, where $U_{i}:\mathcal{H\rightarrow\mathcal{H}}$
and $\alpha_{i}\in\left(0,2\right]$ for each $i=1,2,\dots,m$. By
Theorem \ref{FNE-equiv.cond}, the operator $U_{i}$ is nonexpansive
for each $i=1,2,\dots,m$. Set $\mathcal{M}:=\max_{k\in\mathbb{N}}q_{k}$
and 
\[
\rho_{\left\{ U_{i}\right\} _{i=1}^{m}}:=\min\left\{ \mathcal{M}^{-1}\min_{i\in\left\{ 1,2,\dots,m\right\} }\left(2-\alpha_{i}\right)\alpha_{i}^{-1},1\right\} \le1.
\]
 Let $\left\{ q_{k}\right\} _{k=0}^{\infty}$ be a bounded sequence
of positive integers, $\left\{ \varOmega_{k}\right\} _{k=0}^{\infty}$
be a family of nonempty sets such that $\varOmega_{k}\subset\left\{ 1,2,\dots,m\right\} {}^{\left\{ 1,2,\dots,q_{k}\right\} }.$
That is, $\varOmega_{k}$ is a finite subset of the set of functions
from $\left\{ 1,2,\dots,q_{k}\right\} $ to $\left\{ 1,2,\dots,m\right\} $
for each $k\in\mathbb{N}$. Since the sequence $\left\{ q_{k}\right\} _{k=0}^{\infty}$
is bounded, the number of different elements in the family $\left\{ \varOmega_{k}\right\} _{k=0}^{\infty}$
is finite. For each $k\in\mathbb{N}$ and each $t\in\varOmega_{k}$,
set $V_{k}\left[t\right]:=U_{t\left(q_{k}\right)}\cdots U_{t\left(2\right)}U_{t\left(1\right)}$
and let $\omega_{k}:\varOmega_{k}\rightarrow\left(0,1\right]$ be
a function such that $\sum_{t\in\varOmega_{k}}\omega_{k}\left(t\right)=1$.
For each $k\in\mathbb{N}$, define $T_{\left(\varOmega_{k},\omega_{k}\right)}:=\sum_{t\in\varOmega_{k}}\omega_{k}\left(t\right)V_{k}\left[t\right]$.
Define the set $I:=\limsup_{k\rightarrow\infty}\left\{ T_{\left(\varOmega_{k},\omega_{k}\right)}\right\} $,
that is, $I=\limsup_{k\rightarrow\infty}A_{k}$, where the sequence
of sets $\left\{ A_{k}\right\} _{k=0}^{\infty}$ is defined by singletones,
$A_{k}:=\left\{ T_{\left(\varOmega_{k},\omega_{k}\right)}\right\} $.
Define a family of sets $\left\{ C_{T}\right\} _{T\in I}$ by $C_{T}:=\mathrm{Fix}T$
for each $T\in I$. Set $F:=\cap_{k=0}^{\infty}\mathrm{Fix}T_{\left(\varOmega_{k},\omega_{k}\right)}$
and $C:=\cap_{T\in I}C_{T}$. For each $k\in\mathbb{N}$, we say that
a set $\varOmega_{k}$ is \textit{fit} if 
\[
\cup_{t\in\varOmega_{k}}\im t=\left\{ 1,2,\dots,m\right\} ,
\]
 where $\im t$ denotes the image of the mapping $t$.
\begin{rem}
The operators $T_{\left(\varOmega_{k},\omega_{k}\right)}$, defined
above, are ``string-averaging operators'' as first introduced in
\cite{CEH2001} and further studied in various forms and settings,
see, for instance, Example 5.21 in \cite{BC_book}, \cite{Kong2019}
and \cite{Nikazad2016}, to name but a few. In those and other papers,
the index vector $t$ is called ``a string'', the composite operator
$V_{k}\left[t\right]$ is called ``a string operator'' and $\omega_{k}$
are called ``weight functions''.

We introduce the following definition of \textit{$\limsup$}-admissibility
of sequences of operators.
\end{rem}
\begin{defn}
\label{def:admis}We say that a sequence $\left\{ T_{k}\right\} _{k=0}^{\infty}$of
operators, $T_{k}:\mathcal{H}\rightarrow\mathcal{H}$ for each $k\in\mathbb{N}$,
is \textit{$\limsup$-admissible}, if $\left\{ T_{k}\right\} _{k=0}^{\infty}\subset\limsup_{k\rightarrow\infty}\left\{ T_{k}\right\} $
and for each $T\in\limsup_{k\rightarrow\infty}\left\{ T_{k}\right\} $,
there is an integer $M_{T}>0$ such that $T\in\cup_{n=k}^{k+M_{T}-1}\left\{ T_{\left(\varOmega_{n},w_{n}\right)}\right\} $
for all $k\in\mathbb{N}$.
\end{defn}
\begin{rem}
\label{rem:admiss}Clearly, for each $k_{0}\in\mathbb{N}$, $\limsup_{k\rightarrow\infty}\left\{ T_{k}\right\} =\limsup_{k\rightarrow\infty}\left\{ T_{k_{0}+k}\right\} $
and if a sequence $\left\{ T_{k}\right\} _{k=0}^{\infty}$ of operators
is $\limsup$-admissible, then the sequence $\left\{ T_{k_{0}+k}\right\} _{k=0}^{\infty}$
is $\limsup$-admissible. We use this observation in the sequel.
\end{rem}
\begin{rem}
Observe that if, in the above setting, we require the sequence $\left\{ \omega_{k}\right\} _{k=0}^{\infty}$
to attain a finite number of values, then the sequence $\left\{ T_{\left(\varOmega_{k},\omega_{k}\right)}\right\} _{k=0}^{\infty}$
will also attain a finite number of values. Thus, in order to ensure
the existence of $k_{0}\in\mathbb{N}$ such that the sequence $\left\{ T_{\left(\varOmega_{k_{0}+k},\omega_{k_{0}+k}\right)}\right\} _{k=0}^{\infty}$
is $\limsup$-admissible, we only need to require the existence of
an integer $M_{T}>0$, for each $T\in I$, such that $T\in\cup_{n=k}^{n+M_{T}-1}\left\{ T_{\left(\varOmega_{n},w_{n}\right)}\right\} $
for all $k\in\mathbb{N}$.
\end{rem}

\subsection{\label{subsec:The-convergent-and}The convergent and bounded perturbation
resilient GDSA method}

In this subsection we present several lemmata leading to our main
result (Theorem \ref{main_res} below) that gives conditions under
which the GDSA method converges and is bounded perturbations resilient.
All notions and notations are those presented in Subsection \ref{subsec:Notions,-notations-and}
above.

We consider the following algorithm which is actually Algorithm \ref{alg 1}
with respect to the operators $\left\{ T_{\left(\varOmega_{k},\omega_{k}\right)}\right\} _{k=0}^{\infty}$
which were defined in the settings of the previous subsection.
\begin{lyxalgorithm}
[The General Dynamic String-Averaging (GDSA) algorithm]\label{alg:GDSA}

Given $\varepsilon\in\left(0,1\right]$, $x^{0}\in\mathcal{H}$ and
a sequence $\left\{ T_{\left(\varOmega_{k},\omega_{k}\right)}\right\} _{k=0}^{\infty}$
of operators, the algorithm is defined by the recurrence
\[
x^{k+1}:=x^{k}+\lambda_{k}\left(T_{\left(\varOmega_{k},\omega_{k}\right)}\left(x^{k}\right)-x^{k}\right),
\]
where $\lambda_{k}\in\left[\varepsilon,1+\rho_{\left\{ U_{i}\right\} _{i=1}^{m}}-\varepsilon\right]$
for each $k\in\mathbb{N}$.
\end{lyxalgorithm}
Next we show that under \textit{$\limsup$}-admissibility Algorithm
\ref{alg:GDSA} is strongly coherent.
\begin{lem}
\label{lem:st}Assume that $\left\{ T_{\left(\varOmega_{k},\omega_{k}\right)}\right\} _{k=0}^{\infty}$
is $\limsup$-admissible and let $x_{0}\in\mathcal{H}$. Then $C=F$
and the sequence generated by Algorithm \ref{alg:GDSA} is strongly
coherent and hence coherent.
\end{lem}
\begin{proof}
Definitely, $C=F$. Define a sequence $\left\{ I_{k}\right\} _{k=0}^{\infty}$
of singletons of the elements from the set $I$ by $I_{k}:=\left\{ T_{\left(\varOmega_{k},\omega_{k}\right)}\right\} $
for each $k\in\mathbb{N}$. Since for each $k\in\mathbb{N}$ and each
$T\in I$, we have $T=T_{\left(\varOmega_{m},\omega_{m}\right)}$
for some $m\in\left\{ k,k+1,\dots,k+M_{T}-1\right\} $, it follows
that each $T\in I$ satisfies $T\in\cup{}_{n=k}^{k+M_{T}-1}I_{n}$
for all $k\in\mathbb{N}$. Now assume that $T\in I$, let $z\in\mathcal{H}$,
let $\left\{ z^{k}\right\} _{k=0}^{\infty}$ be a bounded sequence
in $\mathcal{H}$, satisfying
\begin{equation}
z^{k+1}-z^{k}\rightarrow0\,\,\mathrm{and\,\,}T_{\left(\varOmega_{k},\omega_{k}\right)}\left(z^{k}\right)-z^{k}\rightarrow0,\label{eq:}
\end{equation}
and let $\left\{ n_{k}\right\} _{k=0}^{\infty}$ be a strictly increasing
sequence of natural numbers such that $z^{n_{k}}\rightharpoonup z\in\mathcal{H}$
and $T\in I_{n_{k}}$ for all $k\in\mathbb{N}$. Clearly, $T_{\left(\varOmega_{n_{k}},\omega_{n_{k}}\right)}=T$
for all $k\in\mathbb{N}$. By \eqref{eq:}, we have ${\color{blue}T}\mathinner{\color{blue}\left(z^{n_{k}}\right)}-z^{n_{k}}\rightarrow0$
and since $z^{n_{k}}\rightharpoonup z$, it follows from the weak
regularity of the operator $T$ (see Theorem \ref{thm:2.2.5}, Theorem
\ref{FNE-equiv.cond}, Remark \ref{conv_com_comp} and Lemma \ref{lemma:3.2.5}
above) that $z\in\mathrm{Fix}T=C_{T}$. This yields that the condition
in \eqref{eq:-10} holds. Thus, the sequence $\left\{ T_{\left(\varOmega_{k},\omega_{k}\right)}\right\} _{k=0}^{\infty}$
is strongly coherent by Theorem \ref{thm feas} and hence coherent.
This completes the proof of the lemma.
\end{proof}
The next result tells about the $\lambda_{k}$-relaxation of $T_{\left(\varOmega_{k},\omega_{k}\right)}$.
\begin{lem}
\label{lem:relax_output}Let $\left\{ \lambda_{k}\right\} _{k=0}^{\infty}$
be a sequence of real numbers such that $\lambda_{k}\in\left[\varepsilon,1+\rho_{\left\{ U_{i}\right\} _{i=1}^{m}}-\varepsilon\right]$
for each $k\in\mathbb{N}$, where $\varepsilon>0$. Then there exists
a sequence $\left\{ S_{k}\right\} _{k=0}^{\infty}$ of firmly nonexpansive
operators, $S_{k}:\mathcal{H\rightarrow\mathcal{H}}$ for each $k\in\mathbb{N}$,
such that $T_{\left(\varOmega_{k},\omega_{k}\right)\lambda_{k}}$,
the $\lambda_{k}$-relaxation of $T_{\left(\varOmega_{k},\omega_{k}\right)}$,
satisfies
\begin{equation}
T_{\left(\varOmega_{k},\omega_{k}\right)\lambda_{k}}=Id+2\lambda_{k}\left(1+\rho_{\left\{ U_{i}\right\} _{i=1}^{m}}\right)^{-1}\left(S_{k}-Id\right),\label{eq:-23}
\end{equation}
where 
\begin{equation}
2\lambda_{k}\left(1+\rho_{\left\{ U_{i}\right\} _{i=1}^{m}}\right)^{-1}\in\left[2\varepsilon\left(1+\rho_{\left\{ U_{i}\right\} _{i=1}^{m}}\right)^{-1},2-2\varepsilon\left(1+\rho_{\left\{ U_{i}\right\} _{i=1}^{m}}\right)^{-1}\right].\label{eq:-8}
\end{equation}
for each $k\in\mathbb{N}$. Consequently, the operator $T_{\left(\varOmega_{k},\omega_{k}\right)\lambda_{k}}$
is $\left(1+\rho_{\left\{ U_{i}\right\} _{i=1}^{m}}-\lambda_{k}\right)\lambda_{k}^{-1}$-firmly
nonexpansive and, in particular, nonexpansive for each $k\in\mathbb{N}$.
\end{lem}
\begin{proof}
By the definition of $\left\{ \lambda_{k}\right\} _{k=0}^{\infty}$,
\eqref{eq:-8} holds. Define a sequence $\left\{ S_{k}\right\} _{k=0}^{\infty}$
of operators by 
\begin{equation}
S_{k}:=Id+2^{-1}\left(1+\rho_{\left\{ U_{i}\right\} _{i=1}^{m}}\right)\left(T_{\left(\varOmega_{k},\omega_{k}\right)}-Id\right)\label{eq:-31}
\end{equation}
for each $k\in\mathbb{N}$. Set $\rho:=\min_{i\in\left\{ 1,2,\dots,m\right\} }\left(2-\alpha_{i}\right)\alpha_{i}^{-1}$.
By Theorem \ref{FNE_relax_FNE}, $U_{i}$ is $\rho$-firmly nonexpansive
for each $i=1,2,\dots,m$. By Corollary \ref{comp_conv_comb_firm}\textit{(ii)},
for each $k\in\mathbb{N}$, the operator $V_{k}\left(t\right)$ is
$\rho_{\left\{ U_{i}\right\} _{i=1}^{m}}$-firmly nonexpansive for
each $t\in\varOmega_{k}$. As a result, by Corollary \ref{comp_conv_comb_firm}\textit{(i)},
$T_{\left(\varOmega_{k},\omega_{k}\right)}$ is $\rho_{\left\{ U_{i}\right\} _{i=1}^{m}}$-firmly
nonexpansive for each $k\in\mathbb{N}$ and by Corollary \ref{cor: frm-nexp}
and by \eqref{eq:-31}, $S_{k}$ is firmly nonexpansive for each $k\in\mathbb{\mathbb{N}}$.
Now $T_{\left(\varOmega_{k},\omega_{k}\right)\lambda_{k}}$ satisfies
\begin{align}
T_{\left(\varOmega_{k},\omega_{k}\right)\lambda_{k}} & =Id+\lambda_{k}\left(T_{\left(\varOmega_{k},\omega_{k}\right)}-Id\right)=Id+2\lambda_{k}\left(1+\rho_{\left\{ U_{i}\right\} _{i=1}^{m}}\right)^{-1}\left(S_{k}-Id\right)\label{eq:-32}
\end{align}
for each $k\in\mathbb{N}$, so \eqref{eq:-23} holds. Hence, by \eqref{eq:-32}
and \eqref{eq:-8}, $T_{\left(\varOmega_{k},\omega_{k}\right)\lambda_{k}}$
is a $2\lambda_{k}\left(1+\rho_{\left\{ U_{i}\right\} _{i=1}^{m}}\right)^{-1}$-relaxed
firmly nonexpansive operator for each $k\in\mathbb{N}$. By Theorem
\ref{FNE_relax_FNE}, $T_{\left(\varOmega_{k},\omega_{k}\right)\lambda_{k}}$
is $\left(1+\rho_{\left\{ U_{i}\right\} _{i=1}^{m}}-\lambda_{k}\right)\lambda_{k}^{-1}$-firmly
nonexpansive and, in particular, nonexpansive for each $k\in\mathbb{N}$,
and the lemma is proved.
\end{proof}
The next theorem is the main result of this paper. It gives conditions
under which our GDSA algorithm converges and is a bounded perturbations
resilient method. Analogous results in a somewhat more general framework
for the consistent case were presented in \cite{MSA}, where assumptions
of a similar nature were made on the input operators of a certain
procedure. This was possible due to the existence of a common fixed
point of those input operators. Since in this work we consider the
inconsistent case, we need to rely on the assumptions regarding the
output operators of our GDSA procedure.
\begin{thm}
\label{main_res}Assume that $C\not=\emptyset$ and that the sequence
$\left\{ T_{\left(\varOmega_{k},\omega_{k}\right)}\right\} _{k=0}^{\infty}$
is $\limsup$-admissible. Let $\left\{ \lambda_{k}\right\} _{k=0}^{\infty}$
be a sequence of real numbers such that $\lambda_{k}\in\left[\varepsilon,1+\rho_{\left\{ U_{i}\right\} _{i=1}^{m}}-\varepsilon\right]$
for each $k\in\mathbb{N}$, where $\varepsilon>0$, and let $x^{0}\in\mathcal{H}$.
Suppose that $\left\{ x^{k}\right\} _{k=0}^{\infty}$ is a sequence
generated by Algorithm \ref{alg:GDSA} with respect to the sequence
$\left\{ \lambda_{k}\right\} _{k=0}^{\infty}$. Then the following
assertions hold:
\begin{enumerate}
\item The sequence $\left\{ x^{k}\right\} _{k=0}^{\infty}$ converges weakly
to a point $x\in C$.
\item The sequence$\left\{ x^{k}\right\} _{k=0}^{\infty}$ is Fej{\'e}r
monotone with respect to $C$, namely,
\[
\left\Vert x^{k+1}-z\right\Vert ^{2}\le\left\Vert x^{k}-z\right\Vert ^{2}-\varepsilon\left(1+\rho_{\left\{ U_{i}\right\} _{i=1}^{m}}-\varepsilon\right)^{-1}\left\Vert \left(x^{k+1}-x^{k}\right)\right\Vert ^{2}
\]
for each $k\in\mathbb{N}$, and, as a result, $\left\Vert x^{k+1}-x^{k}\right\Vert \rightarrow0$.
\item If each $T\in I$ is approximately shrinking and the family $\left\{ C_{T}\right\} _{T\in I}$
is boundedly regular, then $\left\{ x^{k}\right\} _{k=0}^{\infty}$
converges strongly to a point $x\in C$.
\item The sequence$\left\{ x^{k}\right\} _{k=0}^{\infty}$ is weakly (strongly,
if the convergence in (i) is strong) bounded perturbations resilient
with respect to $C$.
\end{enumerate}
\end{thm}
\begin{proof}
\textit{(i)} By Lemma \ref{lem:st}, $F=C\not=\emptyset$ and the
sequence $\left\{ T_{\left(\varOmega_{k},\omega_{k}\right)}\right\} _{k=0}^{\infty}$
is strongly coherent. By Lemma \ref{lem:relax_output}, there exists
a sequence $\left\{ S_{k}\right\} _{k=0}^{\infty}$ of firmly nonexpansive
operators, $S_{k}:\mathcal{H\rightarrow\mathcal{H}}$ for each $k\in\mathbb{N}$,
such that \eqref{eq:-23} and \eqref{eq:-8} hold for each $k\in\mathbb{N}$.
Hence the sequence $\left\{ S_{k}\right\} _{k=0}^{\infty}$ is also
strongly coherent. Now we apply \eqref{eq:-23}, \eqref{eq:-8}, Remark
\ref{Relaxation has the same Fix}, Theorem \ref{thm:2.2.5} and Theorem
\ref{Main BC} to deduce that the sequence $\left\{ x^{k}\right\} _{k=0}^{\infty}$,
generated by Algorithm \ref{alg:GDSA}, with respect to the sequence
$\left\{ \lambda_{k}\right\} _{k=0}^{\infty}\subset\left[\varepsilon,1+\rho_{\left\{ U_{i}\right\} _{i=1}^{m}}-\varepsilon\right]$
converges weakly to a point $x\in C$.

\textit{(ii)} Let $k\in\mathbb{N}$ be a natural number and $z\in C$.
By the definition of $\left\{ x^{k}\right\} _{k=0}^{\infty}$, Remark
\ref{Relaxation has the same Fix}, since $C\subset\mathrm{Fix}T_{\left(\varOmega_{k},\omega_{k}\right)\lambda_{k}}$
and since (by Lemma \ref{lem:relax_output}) $T_{\left(\varOmega_{k},\omega_{k}\right)\lambda_{k}}$
is $\left(1+\rho_{\left\{ U_{i}\right\} _{i=1}^{m}}-\lambda_{k}\right)\lambda_{k}^{-1}$-firmly
nonexpansive operator (and in particular, a nonexpansive one), we
have
\begin{align}
\left\Vert x^{k+1}-z\right\Vert ^{2} & =\left\Vert T_{\left(\varOmega_{k},\omega_{k}\right)\lambda_{k}}\left(x^{k}\right)-T_{\left(\varOmega_{k},\omega_{k}\right)\lambda_{k}}z\right\Vert ^{2}\le\left\Vert x^{k}-z\right\Vert ^{2}\nonumber \\
 & -\left(1+\rho_{\left\{ U_{i}\right\} _{i=1}^{m}}-\lambda_{k}\right)\lambda_{k}^{-1}\left\Vert \left(x^{k}-T_{\left(\varOmega_{k},\omega_{k}\right)\lambda_{k}}\left(x^{k}\right)\right)\right\Vert ^{2}\label{eq:-9}\\
 & \le\left\Vert x^{k}-z\right\Vert ^{2}-\varepsilon\left(1+\rho_{\left\{ U_{i}\right\} _{i=1}^{m}}-\varepsilon\right)^{-1}\left\Vert \left(x^{k+1}-x^{k}\right)\right\Vert ^{2}.\nonumber 
\end{align}
As a result, $\left\{ x^{k}\right\} _{k=0}^{\infty}$ is Fej{\'e}r
monotone with respect to $C$. Since the real sequence\\
$\left\{ \left\Vert x^{k}-z\right\Vert ^{2}\right\} _{k=0}^{\infty}$
is monotone decreasing and bounded from below by $0$, it converges
and \eqref{eq:-9} now implies that $\left\Vert x^{k+1}-x^{k}\right\Vert \rightarrow0$.

\textit{(iii)} Assume that each $T\in I$ is approximately shrinking
and that the family $\cap_{T\in I}C_{T}$ is boundedly regular. Let
$T\in I$. Clearly, there is a sequence $\left\{ l_{k}\right\} _{k=0}^{\infty}\subset\left\{ 0,1,\dots,M_{T}-1\right\} $
such that $T=T_{\left(\varOmega_{k+l_{k}},\omega_{k+l_{k}}\right)}$
for each $k\in\mathbb{N}$. Then, 
\[
\lambda_{k+l_{k}}\left(T\left(x^{k+l_{k}}\right)-x^{k+l_{k}}\right)=T_{\lambda_{k+l_{k}}}\left(x^{k+l_{k}}\right)-x^{k+l_{k}}=x^{k+l_{k}+1}-x^{k+l_{k}}
\]
for each $k\in\mathbb{N}$. This, combined with \textit{(ii)} implies,
in its turn (since\\
$\left\{ \lambda_{k+l_{k}}\right\} _{k=0}^{\infty}\subset\left[\varepsilon,1+\rho_{\left\{ U_{i}\right\} _{i=1}^{m}}-\varepsilon\right]$),
that $\lim_{k\rightarrow\infty}\left\Vert T\left(x^{k+l_{k}}\right)-x^{k+l_{k}}\right\Vert =0$.
Since $T$ is approximately shrinking and the sequence $\left\{ x^{k+l_{k}}\right\} _{k=0}^{\infty}$
is bounded (because it is weakly convergent by \textit{(i)}), it follows
that
\begin{equation}
\lim_{k\rightarrow\infty}d\left(x^{k+l_{k}},C_{T}\right)=0.\label{eq:-33}
\end{equation}
Now for each $k\in\mathbb{N}$,
\begin{equation}
d\left(x^{k},C_{T}\right)\le\left\Vert x^{k}-x^{k+l_{k}}\right\Vert +d\left(x^{k+l_{k}},C_{T}\right)\label{eq:-12}
\end{equation}
and 
\begin{equation}
\left\Vert x^{k}-x^{k+l_{k}}\right\Vert =\left\Vert \sum_{i=0}^{l_{k}-1}x^{k+i}-x^{k+i+1}\right\Vert \le\sum_{i=0}^{l_{k}-1}\left\Vert x^{k+i}-x^{k+i+1}\right\Vert \label{eq:-13}
\end{equation}
(By definition $\sum_{i=0}^{-1}\left\Vert x^{k+i}-x^{k+i+1}\right\Vert :=0$).
Combining \eqref{eq:-33}, \eqref{eq:-12} and \eqref{eq:-13} with
\textit{(ii)}, we get, due to the finite number of summands in \eqref{eq:-13},
$\lim_{k\rightarrow\infty}d\left(x^{k},C_{T}\right)=0$. Now the bounded
regularity of the family $\left\{ C_{T}\right\} _{T\in I}$ implies
that $\lim_{k\rightarrow\infty}d\left(x^{k},C\right)=0$. By \textit{(ii)}
and Theorem \ref{Fejer}, the sequence $\left\{ x^{k}\right\} _{k=0}^{\infty}$
converges strongly to a point $x\in C$.

\textit{(iv)} Clearly, $x^{k+1}=T_{\left(\varOmega_{k},\omega_{k}\right)\lambda_{k}}\left(x^{k}\right)$
for each $k\in\mathbb{N}$. Consider bounded perturbations by letting
$\left\{ \beta_{k}\right\} _{k=0}^{\infty}$ be a sequence of positive
real numbers such that $\sum_{k=0}^{\infty}\beta_{k}<\infty$ and
letting $\{v^{k}\}_{k=0}^{\infty}$ be a bounded sequence in $\mathcal{H}$.
Assume that $y^{0}\in\mathcal{H}$ and consider the sequence $\left\{ y^{k}\right\} _{k=0}^{\infty}$
generated by the iterative process $y^{k+1}:=T_{\left(\varOmega_{k},\omega_{k}\right)\lambda_{k}}(y^{k}+\beta_{k}v^{k})$.
Suppose that $q\in\mathbb{N}$ and $y\in\mathcal{H}$ are arbitrary.
For each $k\in\mathbb{N}$, set $\gamma_{k}:=\beta_{k}\left\Vert v^{k}\right\Vert \in\left[0,\infty\right)$,
$\lambda_{k}^{\prime}:=\lambda_{q+k}\in\left[\varepsilon,1+\rho_{\left\{ U_{i}\right\} _{i=1}^{m}}-\varepsilon\right]$,
$\varOmega_{k}^{\prime}:=\varOmega_{q+k}$ and $\omega_{k}^{\prime}:=\omega_{q+k}$.
Since by Remark \ref{rem:admiss}, the sequence $\left\{ T_{\left(\varOmega_{k}^{\prime},\omega_{k}^{\prime}\right)}\right\} _{k=0}^{\infty}$
is $\limsup$-admissible and 
\[
C=\cap_{T\in\limsup_{k\rightarrow\infty}\left\{ T_{\left(\varOmega_{k},\omega_{k}\right)}\right\} }C_{T}=\cap_{T\in\limsup_{k\rightarrow\infty}\left\{ T_{\left(\varOmega_{k}^{\prime},\omega_{k}^{\prime}\right)}\right\} }C_{T},
\]
it follows from \textit{(i)} that the sequence $\left\{ x^{\prime k}\right\} _{k=0}^{\infty}$
generated by Algorithm \ref{alg:GDSA} with respect to the sequences
$\left\{ \lambda_{k}^{\prime}\right\} _{k=0}^{\infty}$ and $\left\{ T_{\left(\varOmega_{k}^{\prime},\omega_{k}^{\prime}\right)}\right\} _{k=0}^{\infty}$,
where $x^{\prime0}=y$, converges weakly to a point $x^{\prime}\in C$.
Then, for each $k\in\mathbb{N}$, we have 
\[
x^{\prime k+1}=T_{\left(\varOmega_{k}^{\prime},\omega_{k}^{\prime}\right)\lambda_{k}^{\prime}}\cdots T_{\left(\varOmega_{1}^{\prime},\omega_{1}^{\prime}\right)\lambda_{1}^{\prime}}T_{\left(\varOmega_{0}^{\prime},\omega_{0}^{\prime}\right)\lambda_{0}^{\prime}}\left(y\right)=T_{\left(\varOmega_{q+k},\omega_{q+k}\right)\lambda_{q+k}}\cdots T_{\left(\varOmega_{q+1},\omega_{q+1}\right)\lambda_{q+1}}T_{\left(\varOmega_{q},\omega_{q}\right)\lambda_{q}}\left(y\right)
\]
and hence the sequence $\left\{ T_{\left(\varOmega_{q+k},\omega_{q+k}\right)\lambda_{q+k}}\cdots T_{\left(\varOmega_{q+1},\omega_{q+1}\right)\lambda_{q+1}}T_{\left(\varOmega_{q},\omega_{q}\right)\lambda_{q}}\left(y\right)\right\} _{k=0}^{\infty}$
converges weakly to an element of $C$ for any arbitrary $y\in\mathcal{H}$.
Since for each $k\in\mathbb{N}$, the operator $T_{\left(\varOmega_{k},\omega_{k}\right)\lambda_{k}}$
is nonexpansive, by Lemma \ref{lem:relax_output}, we obtain
\begin{align*}
\left\Vert y^{k+1}-T_{\left(\varOmega_{k},\omega_{k}\right)\lambda_{k}}\left(y^{k}\right)\right\Vert  & =\left\Vert T_{\left(\varOmega_{k},\omega_{k}\right)\lambda_{k}}(y^{k}+\beta_{k}v^{k})-T_{\left(\varOmega_{k},\omega_{k}\right)\lambda_{k}}\left(y^{k}\right)\right\Vert \le\beta_{k}\left\Vert v^{k}\right\Vert =\gamma_{k}.
\end{align*}
We have, 
\[
\cap_{k=0}^{\infty}\mathrm{Fix}T_{\left(\varOmega_{k},\omega_{k}\right)\lambda_{k}}=\cap_{k=0}^{\infty}\mathrm{Fix}T_{\left(\varOmega_{k},\omega_{k}\right)}=C
\]
by Remark \ref{Relaxation has the same Fix} and since the sequence
$\left\{ T_{\left(\varOmega_{k},\omega_{k}\right)}\right\} _{k=0}^{\infty}$
is $\limsup$-admissible. Observe that $\sum_{k=0}^{\infty}\gamma_{k}<\infty$
because the sequence $\{v^{k}\}_{k=0}^{\infty}$ is bounded. We now
deduce from Theorem \ref{error} that the sequence $\left\{ y^{k}\right\} _{k=0}^{\infty}$
converges weakly to an element of $C$ as well, proving that the sequence
$\left\{ x^{k}\right\} _{k=0}^{\infty}$ is weakly bounded perturbations
resilient with respect to $C$. If the convergence of $\left\{ x^{k}\right\} _{k=0}^{\infty}$
is strong, then, again by Theorem \ref{error}, $\left\{ x^{k}\right\} _{k=0}^{\infty}$
is strongly bounded perturbations resilient with respect to $C$.
This completes the proof of the theorem.
\end{proof}
\renewcommand\theenumi{(\alph{enumi})}
\begin{rem}
\label{4.10}~
\begin{enumerate}
\item If there exists a $k_{0}\in\mathbb{N}$ such that the sequence $\left\{ T_{\left(\varOmega_{k_{0}+k},\omega_{k_{0}+k}\right)}\right\} _{k=0}^{\infty}$
is $\limsup$-admissible, then by Remark \ref{rem:admiss}, the statements
of Theorem \ref{main_res} remain true with the only following change
in \textit{(ii)}, wherein the sequence $\left\{ x_{k}\right\} _{k=k_{0}}^{\infty}$
(instead of $\left\{ x^{k}\right\} _{k=0}^{\infty}$) is Fej{\'e}r
monotone with respect to $C$, namely,
\[
\left\Vert x^{k+1}-z\right\Vert ^{2}\le\left\Vert x^{k}-z\right\Vert ^{2}-\varepsilon\left(1+\rho_{\left\{ U_{i}\right\} _{i=1}^{m}}-\varepsilon\right)^{-1}\left\Vert \left(x^{k+1}-x^{k}\right)\right\Vert ^{2}
\]
for each each natural $k\ge k_{0}$.
\item In particular, if $U_{i}$ is a $2\left(\mathcal{M}+1\right)^{-1}$-relaxed
firmly nonexpansive (that is, $\mathcal{M}$-firmly nonexpansive by
Theorem \ref{FNE_relax_FNE}) operator for each $i=1,2,\dots,m$,
then $\rho_{\left\{ U_{i}\right\} _{i=1}^{m}}=1$ and $\lambda_{k}\in\left[\varepsilon,2-\varepsilon\right]$
for each $k\in\mathbb{N}$. If $U_{i}$ is a nonexpansive (that is,
$2$-relaxed firmly nonexpansive by Corollary \ref{cor: frm-nexp})
for each $i=1,2,\dots,m$, then $\rho_{\left\{ U_{i}\right\} _{i=1}^{m}}=0$
and $\lambda_{k}\in\left[\varepsilon,1-\varepsilon\right]$ for each
$k\in\mathbb{N}$.
\item If the space $\mathcal{H}$ is of a finite dimension, then the convergence
in \textit{(i)} is strong.
\end{enumerate}
\end{rem}
\renewcommand\theenumi{(\roman{enumi})}
\begin{example}
\label{SimPro}Let $\varepsilon>0$ be a real number. For all $k\in\mathbb{N}$,
set $q_{k}:=1$, $\varOmega_{k}:=\left\{ 1,2,\dots,m\right\} {}^{\left\{ 1,2,\dots,q_{k}\right\} }$
and $\omega_{k}:=\overline{\omega}$ for a fixed $\overline{\omega}:\left\{ 1,2,\dots,m\right\} {}^{\left\{ 1,2,\dots,q_{k}\right\} }\rightarrow\left(0,1\right]$.
Clearly, for each $i\in\left\{ 1,2,\dots,m\right\} $, there is a
unique string ${\color{blue}t^{i}}\in\left\{ 1,2,\dots,m\right\} {}^{\left\{ 1,2,\dots,q_{k}\right\} }$
such that $t^{i}\left(1\right)=i$. Hence we can define the mapping
$\omega:\left\{ 1,2,\dots,m\right\} \rightarrow\left\{ 1,2,\dots,m\right\} {}^{\left\{ 1,2,\dots,q_{k}\right\} }$
by $\omega_{i}:=\overline{\omega}\left(t^{i}\right)$ for each $i\in\left\{ 1,2,\dots,m\right\} $.
In this case Algorithm \ref{alg:GDSA} with the above provisions provides
a fully-simultaneous method, that is, 
\begin{equation}
T_{\left(\varOmega_{k},\omega_{k}\right)}=\sum_{t\in\varOmega_{k}}\overline{\omega}\left(t\right)U_{t\left(1\right)}=\sum_{i=1}^{m}\overline{\omega}\left(t^{i}\right)U_{t^{i}\left(1\right)}=\sum_{i=1}^{m}\omega_{i}U_{i}\label{eq:-34}
\end{equation}
 for each $k\in\mathbb{N}$, and $C=F=\mathrm{Fix}\sum_{i=1}^{m}\omega_{i}U_{i}$.
We see from \eqref{eq:-34} that the sequence $\left\{ T_{\left(\varOmega_{k},\omega_{k}\right)}\right\} _{k=0}^{\infty}$
is $\limsup$-admissible. Hence, under the assumption that $\mathrm{Fix}\sum_{i=1}^{m}\omega_{i}U_{i}\not=\emptyset$
we obtain, by Theorem \ref{main_res}, the weak convergence of this
fully-simultaneous method, with parameters $\left\{ \lambda_{k}\right\} _{k=0}^{\infty}\subset\left[\varepsilon,1+\rho_{\left\{ U_{i}\right\} _{i=1}^{m}}-\varepsilon\right]$,
to a point in $C$. In particular, when $U_{i}=P_{C_{i}}$, where
$C_{i}$ is a nonempty, closed and convex subset of $\mathcal{H}$
for each $i=1,2,\dots,m$ , we obtain the well-known simultaneous
projection method, see, for example, \cite[Subsection 5.4]{C_book}.
In this case $\rho_{\left\{ U_{i}\right\} _{i=1}^{m}}=1$ (by Example
\ref{metric projection} and Remark \ref{4.10}(b)), $I=\left\{ \sum_{i=1}^{m}\omega_{i}P_{C_{i}}\right\} $
and 
\begin{equation}
C=F=\mathrm{Fix}\sum_{i=1}^{m}\omega_{i}P_{C_{i}}=\underset{x\in\mathcal{H}}{\mathrm{Argmin}}f\left(x\right),\label{eq:-5}
\end{equation}
where $f:\mathcal{H}\rightarrow\mathbb{R}$ is a, so called, proximity
function defined by $f:=2^{-1}\sum_{i=1}^{m}\omega_{i}\left\Vert P_{C_{i}}-Id\right\Vert ^{2}$.
For the proof of the last equality in \eqref{eq:-5}, see, for instance,
Theorem 4.4.6 in \cite{C_book}. By Theorem \ref{main_res}\textit{(i)},
the simultaneous projection method with parameters $\left\{ \lambda_{k}\right\} _{k=0}^{\infty}\subset\left[\varepsilon,2-\varepsilon\right]$
converges weakly to a point in $\underset{x\in\mathcal{H}}{\mathrm{Argmin}}f\left(x\right)$.
If, in addition, the operator $\sum_{i=1}^{m}\omega_{i}P_{C_{i}}$
is approximately shrinking, then (since the family $\left\{ C_{T}\right\} _{T\in I}=\left\{ C\right\} $
is boundedly regular), this convergence is strong by Theorem \ref{main_res}\textit{(ii)}.
\end{example}

\section{\label{sec5}Application of the GDSA method to the Superiorization
Methodology}

In this section we introduce the superiorized version of Algorithm
\ref{alg:GDSA} with respect to a convex and continuous objective
function $\phi:\mathcal{H}\rightarrow\mathbb{R}$ and sequences $\left\{ T_{\left(\varOmega_{k},\omega_{k}\right)}\right\} _{n=0}^{\infty}$
and $\left\{ \lambda_{k}\right\} _{k=0}^{\infty}\subset\left[\varepsilon,1+\rho_{\left\{ U_{i}\right\} _{i=1}^{m}}-\varepsilon\right]$,
where $\varepsilon>0$, defined in Section \ref{sec4}. We investigate
its convergence properties in the framework of the Superiorization
Methodology. This provides a generalization of Algorithm 4.1, presented
in \cite{CZ_sup} and applies to the results below concerning the
behavior of the superiorized version of Algorithm \ref{alg:GDSA}.

\subsection{A few words about the Superiorization Methodology\label{subsec:A-few-words}}

In this subsection we recall the brief description of the superiorization
methodology (SM), quoted from the preface to the 2017 special issue
\cite{CHJ-special-issue-2017} on ``Superiorization: Theory and Applications''.
``The superiorization methodology is used for improving the efficacy
of iterative algorithms whose convergence is resilient to certain
kinds of perturbations. Such perturbations are designed to ``force''
the perturbed algorithm to produce more useful results for the intended
application than the ones that are produced by the original iterative
algorithm. The perturbed algorithm is called the ``superiorized version''
of the original unperturbed algorithm. If the original algorithm is
computationally efficient and useful in terms of the application at
hand and if the perturbations are simple and not expensive to calculate,
then the advantage of this method is that, for essentially the computational
cost of the original algorithm, we are able to get something more
desirable by steering its iterates according to the designed perturbations.
This is a very general principle that has been used successfully in
some important practical applications, especially for inverse problems
such as image reconstruction from projections, intensity-modulated
radiation therapy and nondestructive testing, and awaits to be implemented
and tested in additional fields.''

Further information and references on the SM can be found in papers
listed in the bibliographic collection on the dedicated Webpage \cite{SM-bib-page}.
For recent works that include introductory material on the SM see,
for example, \cite{asymmetric-2023}, \cite{erturk-salim-2023}, \cite{humphries-2022}
and \cite{Torregrosa-2024}.

\subsection{The superiorized version of the GDSA algorithm\label{subsec:The-superiorized-version}}

In Subsection \ref{not} we defined $I:=\limsup_{k\rightarrow\infty}\left\{ T_{\left(\varOmega_{k},\omega_{k}\right)}\right\} $
and the family $\left\{ C_{T}\right\} _{T\in I}$ with the intersection
$C:=\cap_{T\in I}C_{T}$. Let $\phi:\mathcal{H}\rightarrow\mathbb{R}$
be a convex and continuous real valued objective function. Define
the constrained minimum set by 
\[
C_{\min}:=\mathrm{Argmin}\{\phi\left(x\right)\mid x\in C\}.
\]

\begin{lyxalgorithm}
\label{supGDSA}[The superiorized version of the GDSA algorithm]
Given $y^{0}\in\mathcal{H}$, a sequence $\left\{ N_{k}\right\} _{k=0}^{\infty}$
of positive integers, a sequence $\left\{ \lambda_{k}\right\} _{k=0}^{\infty}$
of positive numbers and a family of positive real sequences $\left\{ \left\{ \beta_{k,n}\right\} _{n=1}^{N_{k}}\right\} _{k=0}^{\infty}$
such that $\sum_{k=0}^{\infty}\sum_{n=1}^{N_{k}}\beta_{k,n}<\infty$,
the algorithm is defined by the recurrences
\[
y^{k+1}:=y^{k}+\sum_{n=1}^{N_{k}}\beta_{k,n}v^{k,n}+\lambda_{k}\left(T_{\left(\varOmega_{k},\omega_{k}\right)}\left(y^{k}+\sum_{n=1}^{N_{k}}\beta_{k,n}v^{k,n}\right)-y^{k}-\sum_{n=1}^{N_{k}}\beta_{k,n}v^{k,n}\right)
\]
\[
\mathrm{wherein}
\]
\begin{equation}
v^{k,n+1}:=\begin{cases}
-\left\Vert s^{k,n}\right\Vert ^{-1}s^{k,n}, & \mathrm{if}\,\,0\not\in\partial\phi\left(y^{k}+\sum_{i=1}^{n}\beta_{k,i}v^{k,i}\right),\\
0, & \mathrm{if}\,\,0\in\partial\phi\left(y^{k}+\sum_{i=1}^{n}\beta_{k,i}v^{k,i}\right),
\end{cases}\label{eq:-6}
\end{equation}
for each $k\in\mathbb{N}$ and each $n=0,1,\dots,N_{k}-1$, where
$s^{k,n}$ is a selection of the subgradient $\partial\phi\left(y^{k}+\sum_{i=1}^{n}\beta_{k,i}v^{k,i}\right)$
(which exists by Theorem \ref{subdiff_ne+b}) for each $k\in\mathbb{N}$
and each $n=0,1,\dots,N_{k}-1$ (recalling that, by definition, $\sum_{i=1}^{0}\beta_{k,i}v^{k,i}:=0$).
\end{lyxalgorithm}
It is true that, under the assumptions of Theorem \ref{main_res},
where, in particular, $\left\{ \lambda_{k}\right\} _{k=0}^{\infty}\subset\left[\varepsilon,1+\rho_{\left\{ U_{i}\right\} _{i=1}^{m}}-\varepsilon\right]$
for $\varepsilon>0$, the sequence $\left\{ y^{k}\right\} _{k=0}^{\infty}$,
generated by Algorithm \ref{supGDSA} with respect to the sequence
$\left\{ \lambda_{k}\right\} _{k=0}^{\infty}$, satisfies Statements
\textit{(i)} and \textit{(iii)} of Theorem \ref{main_res}. Indeed,
define a sequence $\left\{ \beta_{k}\right\} _{k=0}^{\infty}$ of
positive real numbers by $0<\beta_{k}:=\sum_{n=1}^{N_{k}}\beta_{k,n}$
for each $k\in\mathbb{N}$. Then, $\sum_{k=0}^{\infty}\beta_{k}<\infty$
and we have 
\begin{align}
y^{k+1} & =T_{\left(\varOmega_{k},\omega_{k}\right)\lambda_{k}}\left(y^{k}+\sum_{n=1}^{N_{k}}\beta_{k,n}v^{k,n}\right)=T_{\left(\varOmega_{k},\omega_{k}\right)\lambda_{k}}\left(y^{k}+\beta_{k}\sum_{n=1}^{N_{k}}\beta_{k,n}\beta_{k}^{-1}v^{k,n}\right).\label{eq:-7}
\end{align}
It follows, by the triangle inequality and \eqref{eq:-6} that 
\[
\left\Vert \sum_{n=1}^{N_{k}}\beta_{k,n}\beta_{k}^{-1}v^{k,n}\right\Vert \le\sum_{n=1}^{N_{k}}\beta_{k,n}\beta_{k}^{-1}=1,
\]
that is, the sequence $\left\{ \sum_{n=1}^{N_{k}}\beta_{k,n}\beta_{k}^{-1}v^{k,n}\right\} _{k=0}^{\infty}$
is bounded in $\mathcal{H}$. By using Theorem \ref{main_res}\textit{(iv)}
and \eqref{eq:-7}, we obtain the following corollary.
\begin{cor}
\label{basic cor}Let $y^{0}\in\mathcal{H}$. Assume that the sequence
$\left\{ T_{\left(\varOmega_{k},\omega_{k}\right)}\right\} _{k=0}^{\infty}$
is $\limsup$-admissible and $C\not=\emptyset$. Let $\left\{ \lambda_{k}\right\} _{k=0}^{\infty}$
be a sequence of real numbers such that $\lambda_{k}\in\left[\varepsilon,1+\rho_{\left\{ U_{i}\right\} _{i=1}^{m}}-\varepsilon\right]$
for each $k\in\mathbb{N}$, where $\varepsilon>0$. Then the sequence
$\left\{ y^{k}\right\} _{k=0}^{\infty}$ generated by Algorithm \ref{supGDSA}
converges weakly to a point $y\in C$. If, in addition, each $T\in I$
is approximately shrinking and the family $\left\{ C_{T}\right\} _{T\in I}$
is boundedly regular, then $\left\{ y^{k}\right\} _{k=0}^{\infty}$
converges strongly to a point $y\in C$.
\end{cor}
It turns out that if the sequence $\left\{ y^{k}\right\} _{k=0}^{\infty}$,
generated by Algorithm \ref{supGDSA}, converges strongly to a point
$y\in\mathcal{H}$, then the sequence $\left\{ y^{k}\right\} _{k=k_{0}}^{\infty}$
is strictly Fej{\'e}r monotone with respect to $C_{\mathrm{min}}$
for some $k_{0}\in\mathbb{N}$. In order to show this, we need the
following auxiliary lemma.

\begin{lem}
\label{aux_lem} For an arbitrary nonempty subset $C$ of $\mathcal{H}$
and $y,z\in C$ such that $z\in\mathrm{Argmin}_{x\in C}\phi\left(x\right)$
and $y\not\in\mathrm{Argmin}_{x\in C}\phi\left(x\right)$, there exist
real numbers $r_{1}>0$ and $r_{2}>0$ so that for each $\overline{y}\in B\left(y,r_{1}\right)$
and $v\in\partial\phi\left(\overline{y}\right)$, the following assertions
are satisfied:
\begin{enumerate}
\item $0\not\in\partial\phi\left(\overline{y}\right)$ and for each $\overline{z}\in B\left(z,r_{2}\right)$
\begin{equation}
\left\langle \left\Vert v\right\Vert ^{-1}v,\overline{z}-\overline{y}\right\rangle <0.\label{eq:-7-1}
\end{equation}
\item We have 
\begin{equation}
\left\langle \left\Vert v\right\Vert ^{-1}v,z-\overline{y}\right\rangle <-2^{-1}r_{2}.\label{eq:-16}
\end{equation}
\item Let $p$ be a nonnegative integer. Assume that $\left\{ \alpha_{n}\right\} _{n=1}^{p}$
is a sequence of positive real numbers such that $\sum_{n=1}^{p}a_{n}<2^{-1}r_{1}$
and $\left\{ v^{n}\right\} _{n=1}^{p}\subset\mathcal{H}\backslash\left\{ 0\right\} $
is a sequence such that $v^{n}\in\partial\phi\left(\overline{y}-\sum_{i=1}^{n-1}\alpha_{i}\left\Vert v^{i}\right\Vert ^{-1}v^{i}\right)$
for each $n=1,2,\dots,p$. If, in addition, $\overline{y}\in B\left(y,2^{-1}r_{1}\right)$,
then
\begin{equation}
\left\Vert \overline{y}-\sum_{n=1}^{p}\alpha_{n}\left\Vert v^{n}\right\Vert ^{-1}v^{n}-z\right\Vert ^{2}\le\left\Vert \overline{y}-z\right\Vert ^{2}-\sum_{n=1}^{p}\left(r_{2}-\alpha_{n}\right)\alpha_{n}\label{eq:-25}
\end{equation}
(by definition $\sum_{n=1}^{0}\alpha_{n}\left\Vert v^{n}\right\Vert ^{-1}v^{n}:=\sum_{n=1}^{0}\left(r_{2}-\alpha_{n}\right)\alpha_{n}:=0)$.
\end{enumerate}
\end{lem}
\begin{proof}
Since $z\in\mathrm{Argmin}_{x\in C}\phi\left(x\right)$ and $y\not\in\mathrm{Argmin}_{x\in C}\phi\left(x\right)$,
we observe that $\phi\left(y\right)-\phi\left(z\right)>0$. By the
continuity of $\phi$, there exist $r_{1}>0$ and $r_{2}>0$ such
that
\begin{equation}
\phi\left(\overline{y}\right)-\phi\left(\overline{z}\right)>0,\label{eq:-15}
\end{equation}
for each $\overline{y}\in B\left(y,r_{1}\right)$ and $\overline{z}\in B\left(z,r_{2}\right)$.
Let $\overline{y}\in B\left(y,r_{1}\right)$ and $v\in\partial\phi\left(\overline{y}\right)$.

\textit{(i)} In view of \eqref{eq:-15} and \eqref{eq:-24}, we have
for each $\overline{z}\in B\left(z,r_{2}\right)$,
\[
\left\langle v,\overline{z}-\overline{y}\right\rangle <0.
\]
It follows that $v\not=0$ and \eqref{eq:-7-1} holds. Since $v$
is an arbitrary element of $\partial\phi\left(\overline{y}\right)$,
it follows that $0\not\in\partial\phi\left(\overline{y}\right)$.

\textit{(ii)} Set $\overline{z}:=z+2^{-1}r_{2}\left\Vert v\right\Vert ^{-1}v$.
Then $\overline{z}\in B\left(z,r_{2}\right)$ and by \eqref{eq:-7-1},
\[
\left\langle \left\Vert v\right\Vert ^{-1}v,z+2^{-1}r_{2}\left\Vert v\right\Vert ^{-1}v-\overline{y}\right\rangle =\left\langle \left\Vert v\right\Vert ^{-1}v,\overline{z}-\overline{y}\right\rangle <0
\]
and \eqref{eq:-16} follows.

\textit{(iii)} Assume that $\overline{y}\in B\left(y,2^{-1}r_{1}\right)$.
Since $\sum_{n=1}^{p-1}a_{n}<2^{-1}r_{1}$, we obtain 
\begin{equation}
\left(\overline{y}-\sum_{n=1}^{p-1}\alpha_{n}\left\Vert v^{n}\right\Vert ^{-1}v^{n}\right)\in B\left(y,r_{1}\right).\label{eq:-27}
\end{equation}
It is true that, 
\begin{equation}
v^{p}\in\partial\phi\left(\overline{y}-\sum_{n=1}^{p-1}\alpha_{n}\left\Vert v^{n}\right\Vert ^{-1}v^{n}\right).\label{eq:-28}
\end{equation}
The proof of \eqref{eq:-25} is by induction on $p$. Clearly, \eqref{eq:-25}
is true for $p=0$. Assume that $p>0$. Then by the induction hypothesis,
\eqref{eq:-27}, \eqref{eq:-28} and \textit{(ii)} above,
\begin{align*}
\left\Vert \overline{y}-\sum_{n=1}^{p}\alpha_{n}\left\Vert v^{n}\right\Vert ^{-1}v^{n}-z\right\Vert ^{2} & =\left\Vert \overline{y}-\sum_{n=1}^{p-1}\alpha_{n}\left\Vert v^{n}\right\Vert ^{-1}v^{n}-\alpha_{p}\left\Vert v^{p}\right\Vert ^{-1}v^{p}-z\right\Vert ^{2}\\
 & =\left\Vert \overline{y}-\sum_{n=1}^{p-1}\alpha_{n}\left\Vert v^{n}\right\Vert ^{-1}v^{n}-z\right\Vert ^{2}\\
 & +2\alpha_{p}\left\langle \left\Vert v^{p}\right\Vert ^{-1}v^{p},z-\left(\overline{y}-\sum_{n=1}^{p-1}\alpha_{n}\left\Vert v^{n}\right\Vert ^{-1}v^{n}\right)\right\rangle +\alpha_{p}^{2}\\
 & \le\left\Vert \overline{y}-z\right\Vert ^{2}-\sum_{n=1}^{p-1}\left(r_{2}-\alpha_{n}\right)\alpha_{n}-\alpha_{p}r_{2}+\alpha_{p}^{2}\\
 & =\left\Vert \overline{y}-z\right\Vert ^{2}-\sum_{n=1}^{p}\left(r_{2}-\alpha_{n}\right)\alpha_{n}.
\end{align*}
Lemma \ref{aux_lem} is now proved.
\end{proof}
The following ``theorem of alternatives'' provides a more general
analogue of Theorem 4.1 in \cite{CZ_sup} in the setting of our superiorized
version of the GDSA algorithm.
\begin{thm}
\label{Strict Fejer} Let $y^{0}\in\mathcal{H}$ and assume that the
sequence $\left\{ y^{k}\right\} _{k=0}^{\infty}$, generated by Algorithm
\ref{supGDSA} with respect to the sequence $\left\{ \lambda_{k}\right\} _{k=0}^{\infty}\subset\left[\varepsilon,1+\rho_{\left\{ U_{i}\right\} _{i=1}^{m}}-\varepsilon\right]$,
where $\varepsilon>0$, converges strongly to a point $y\in C$. Then
exactly one of the following two alternatives holds:
\begin{enumerate}
\item $y\in C_{\mathrm{min}}$.

\suspend{enumerate} ~~or \resume{enumerate}
\item $y\not\in C_{\mathrm{min}}$ and there exists $k_{0}\in\mathbb{N}$
such that $\left\{ y^{k}\right\} _{k=k_{0}}^{\infty}$ is strictly
Fej{\'e}r monotone with respect to $C_{\mathrm{min}}$. Namely, there
exists a sequence $\left\{ u_{k}\right\} _{k=k_{0}}^{\infty}$ of
positive real numbers such that $\left\Vert y^{k+1}-z\right\Vert ^{2}\le\left\Vert y^{k}-z\right\Vert ^{2}-u_{k}$
for every $z\in C_{\mathrm{min}}$ and for all natural $k\ge k_{0}$.
\end{enumerate}
\end{thm}
\begin{proof}
Assume that $\left\{ y^{k}\right\} _{k=0}^{\infty}$ converges strongly
to a point $y\not\in C_{\mathrm{min}}$. Then, $y\in C$ by Corollary
\ref{basic cor} and $y\not\in\mathrm{Argmin}_{x\in C}\phi\left(x\right)$.
Assume that $z\in C_{\mathrm{min}}$. By Lemma \ref{aux_lem}, there
exist real numbers $r_{1}>0$ and $r_{2}>0$ such that each $\overline{y}\in B\left(y,r_{1}\right)$
and $v\in\partial\phi\left(\overline{y}\right)$ satisfy its assertions.
By using the strong convergence of $\left\{ y^{k}\right\} _{k=0}^{\infty}$
to $y$ and the convergence of the series $\sum_{k=0}^{\infty}\sum_{n=1}^{N_{k}}\beta_{k,n}$,
choose $k_{0}\in\mathbb{N}$ such that 
\begin{equation}
y^{k}\in B\left(y,2^{-1}r_{1}\right)\label{eq:-30}
\end{equation}
 and 
\begin{equation}
\sum_{n=1}^{N_{k}}\beta_{k,n}<\min\left\{ 2^{-1}r_{1},r_{2}\right\} \label{eq:-14}
\end{equation}
for each integer $k\ge k_{0}$. This yields, for each $k\ge k_{0}$,
\[
y^{k}+\sum_{i=1}^{n-1}\beta_{k,i}v^{k,i}\in B\left(y,r_{1}\right)
\]
for each $n=1,2,\dots,N_{k}$, and, consequently, by Lemma \ref{aux_lem}\textit{(i)},
\begin{equation}
0\not\in\partial\phi\left(y^{k}+\sum_{i=1}^{n-1}\beta_{k,i}v^{k,i}\right)\label{eq:-17}
\end{equation}
for each $n=1,2,\dots,N_{k}$. Let $k\ge k_{0}$ be an integer. By
\eqref{eq:-6} and \eqref{eq:-17}, 
\begin{equation}
v^{k,n}=-\left\Vert s^{k,n-1}\right\Vert ^{-1}s^{k,n-1},\label{eq:-29}
\end{equation}
where 
\begin{equation}
s^{k,n-1}\in\partial\phi\left(y^{k}+\sum_{i=1}^{n-1}\beta_{k,i}v^{k,i}\right),\label{eq:-26}
\end{equation}
for each $n=1,2,\dots,N_{k}$. Set $p:=N_{k}$ and $\overline{y}:=y^{k}$.
For each $n=1,2,\dots,p$, define $\alpha_{n}:=\beta_{k,n}>0$ and
$v^{n}:=s^{k,n-1}$. Then, by \eqref{eq:-30}, \eqref{eq:-14}, \eqref{eq:-17},
\eqref{eq:-29} and \eqref{eq:-26}, $\overline{y}\in B\left(y,2^{-1}r_{1}\right)$,
$\sum_{n=1}^{p}\alpha_{n}<2^{-1}r_{1}$, $\left\{ v^{n}\right\} _{n=1}^{p}\subset\mathcal{H}\backslash\left\{ 0\right\} $
and $v^{n}\in\partial\phi\left(\overline{y}-\sum_{i=1}^{n-1}\alpha_{i}\left\Vert v^{i}\right\Vert ^{-1}v^{i}\right)$
for each $n=1,2,\dots,p$. Since the operator $T_{\left(\varOmega_{k},\omega_{k}\right)\lambda_{k}}$,
the $\lambda_{k}$-relaxation of $T_{\left(\varOmega_{k},\omega_{k}\right)}$,
is nonexpansive, by Lemma \ref{lem:relax_output}, we obtain from
Lemma \ref{aux_lem}\textit{(iii),} that
\begin{align*}
\left\Vert y^{k+1}-z\right\Vert ^{2} & =\left\Vert T_{\left(\varOmega_{k},\omega_{k}\right)\lambda_{k}}\left(y^{k}+\sum_{n=1}^{N_{k}}\beta_{k,n}v^{k,n}\right)-z\right\Vert ^{2}\le\left\Vert y^{k}+\sum_{n=1}^{N_{k}}\beta_{k,n}v^{k,n}-z\right\Vert ^{2}\\
 & =\left\Vert \overline{y}-\sum_{n=1}^{p}\alpha_{n}\left\Vert v_{n}\right\Vert ^{-1}v_{n}-z\right\Vert ^{2}\le\left\Vert \overline{y}-z\right\Vert ^{2}-\sum_{n=1}^{p}\left(r_{2}-\alpha_{n}\right)\alpha_{n}\\
 & =\left\Vert y^{k}-z\right\Vert ^{2}-\sum_{n=1}^{N_{k}}\left(r_{2}-\beta_{k,n}\right)\beta_{k,n}.
\end{align*}
Now set $u_{k}:=\sum_{n=1}^{N_{k}}\left(r_{2}-\beta_{k,n}\right)\beta_{k,n}$
for each natural $k\ge k_{0}$. Then since $\sum_{n=1}^{N_{k}}\beta_{k,n}<r_{2}$,
by \eqref{eq:-14}, the result follows and the proof of the theorem
is complete.
\end{proof}
\begin{rem}
Note that we do not assume any admissibility condition on the sequence
of operators $\left\{ T_{\left(\varOmega_{k},\omega_{k}\right)}\right\} _{k=0}^{\infty}$
in Theorem \ref{Strict Fejer}.

Combining Theorem \ref{Strict Fejer} with Corollary \ref{basic cor},
we obtain the following corollary.
\begin{cor}
Let $y^{0}\in\mathcal{H}$ and assume that the sequence $\left\{ T_{\left(\varOmega_{k},\omega_{k}\right)}\right\} _{k=0}^{\infty}$
is $\limsup$-admissible. Suppose also that each $T\in I$ is approximately
shrinking and the family $\left\{ C_{T}\right\} _{T\in I}$ is boundedly
regular. Then the sequence $\left\{ y^{k}\right\} _{k=0}^{\infty}$
generated by Algorithm \ref{supGDSA} with respect to the sequence
$\left\{ \lambda_{k}\right\} _{k=0}^{\infty}\subset\left[\varepsilon,1+\rho_{\left\{ U_{i}\right\} _{i=1}^{m}}-\varepsilon\right]$,
where $\varepsilon>0$, converges strongly to a point $y\in C$ and
exactly one of the following two alternatives holds:
\begin{enumerate}
\item $y\in C_{\mathrm{min}}$

\suspend{enumerate} ~~or \resume{enumerate}
\item $y\not\in C_{\mathrm{min}}$ and there exists $k_{0}\in\mathbb{N}$
such that $\left\{ y^{k}\right\} _{k=k_{0}}^{\infty}$ is strictly
Fej{\'e}r monotone with respect to $C_{\mathrm{min}}$. Namely, there
exists a sequence $\left\{ u_{k}\right\} _{k=k_{0}}^{\infty}$ of
positive real numbers such that for every $z\in C_{\mathrm{min}}$,
$\left\Vert y^{k+1}-z\right\Vert ^{2}\le\left\Vert y^{k}-z\right\Vert ^{2}-u_{k}$
for all natural $k\ge k_{0}$.
\end{enumerate}
\end{cor}
\end{rem}

\section{\label{sec:Conclusion}Conclusion}

In this paper we proposed and investigated a General Dynamic String-Averaging
(GDSA) iterative scheme in the inconsistent case. The main tool is
the property called ``strong coherence'' which serves as a sufficient
condition for convergence of iterative schemes governed by infinite
sequences of operators. The GDSA algorithm is bounded perturbation
resilient and, as such, we applied to it the superiorization methodology
and derived for the superiorized version of the GDSA algorithm a ``theorem
of alternatives'' proving strict Fej{\'e}r monotonicity with respect
to the minimum set of the underlying constrained minimization problem
data.

\subsubsection*{Data availability}

No data was used for the research described in the article.

\subsubsection*{Acknowledgments}

We thank Dr. Niklas Wahl and Tobias Becher from the German Cancer
Research Center (DKFZ) in Heidelberg, Germany, for enlightening conversations
about the feasibility-seeking model in radiation therapy treatment
planning which initiated our work on this paper. We are grateful to
two anonymous referees for valuable comments  that helped in improving
our work. This work is supported by U.S. National Institutes of Health
(NIH) Grant Number R01CA266467 and by the Cooperation Program in Cancer
Research of the German Cancer Research Center (DKFZ) and Israel's
Ministry of Innovation, Science and Technology (MOST).

\bibliographystyle{JOTA_style}
\bibliography{bank_of_references}

\end{document}